\documentclass[a4paper,oneside,11pt]{article}

\usepackage{hyperref}  
\usepackage[a4paper]{geometry}
\usepackage{aeguill}                 
\usepackage{graphicx}
\usepackage{amsmath,amsfonts,amssymb,amsthm}
\usepackage{pstricks,comment} 
\usepackage{epstopdf,enumerate}
\usepackage{srcltx}
\usepackage[utf8]{inputenc} 
\usepackage[T1]{fontenc} 
 
\usepackage[francais,english]{babel}

\title{Boundary amenability of $\Out(F_N)$}
\author{Mladen Bestvina, Vincent Guirardel and Camille Horbez}

\usepackage[all]{xy}
\usepackage{bbm}

\newcommand{\calk}{\mathcal{K}}
\newcommand{\Prob}{\text{Prob}}
\newcommand{\baro}{\overline{\calo}}

\newcommand{\ra}{\rightarrow}
\newcommand{\m}{^{-1}}
\newcommand{\dunion}{\sqcup}
\newcommand{\eps}{\varepsilon}
\renewcommand{\epsilon}{\varepsilon}
\newcommand{\calr}{\mathcal{R}}
\newcommand{\calf}{\mathcal{F}}

\newcommand{\calt}{\mathcal{T}}

\newcommand{\caly}{\mathcal{Y}}
\newcommand{\calb}{\mathcal{B}}

\newcommand{\calp}{\mathcal{P}}
\newcommand{\cala}{\mathcal{A}}

\newcommand{\calo}{\mathcal{O}}

\newcommand{\cals}{\mathcal{S}}

\newcommand{\bbP}{\mathbb{P}}
\newcommand{\bbQ}{\mathbb{Q}}
\newcommand{\bbR}{\mathbb{R}}
\newcommand{\bbN}{\mathbb{N}}
\newcommand{\bbZ}{\mathbb{Z}}
\newcommand{\actson}{\curvearrowright}
\newcommand{\es}{\emptyset}

\makeatletter
\edef\@tempa#1#2{\def#1{\mathaccent\string"\noexpand\accentclass@#2 }}
\@tempa\rond{017}
\makeatother
\newcommand{\cale}{\mathcal{E}}

\newcommand{\Out}{\mathrm{Out}}
\newcommand{\Aut}{\mathrm{Aut}}
\newcommand{\Inn}{\mathrm{Inn}}
\newcommand{\Simp}{\mathrm{Simp}}
\newcommand{\NonG}{\mathcal{NG}eom}
\newcommand{\Os}{\calo} 
\newcommand{\AT}{\mathcal{AT}}
\newcommand{\ATf}{\mathcal{AT}^{\mathrm{free}}}
\newcommand{\Geom}{{\mathcal{G}eom}}
\newcommand{\GI}{{\mathcal{GI}}}
\newcommand{\Mor}{\text{Mor}}
\newcommand{\Opt}{\text{Opt}}
\newcommand{\VOpt}{\text{VOpt}}

\newcommand{\MCG}{\mathrm{MCG}}
\newcommand{\calz}{\mathcal{Z}}
\newcommand{\calu}{\mathcal{U}}
\newcommand{\BBT}{\mathrm{BBT}}
\newcommand{\covol}{\mathrm{covol}}
\newcommand{\vol}{\mathrm{vol}}

\newtheorem{de}{Definition} [section]
\newtheorem{theo}[de]{Theorem} 
\newtheorem{prop}[de]{Proposition}
\newtheorem{lemma}[de]{Lemma}
\newtheorem{cor}[de]{Corollary}

\newtheorem*{fact*}{Fact}

\theoremstyle{remark}
\newtheorem{rk}[de]{Remark}

\newtheorem{question}[de]{Question}

\newcommand{\turn}[1]{{#1}^{\mathrm{turn}}}
\newcommand{\Name}{\mathrm{Name}}
\newcommand{\oturn}{\baro^{\mathrm{turn}}}
\newcommand{\FGeom}{\mathcal{FG}eom}
\newcommand{\ad}{\text{ad}}

\newcommand{\amen}{\mathrm{amen}}

\newcommand{\ZS}{\calz\mathrm{S}}

\begin{document}
\maketitle

\normalsize

\begin{flushright}
\emph{Kaum nennt man die Dinge beim richtigen Namen, so verlieren sie ihren gefährlichen Zauber.}

Elias Canetti
\end{flushright}

\begin{flushright}
  \emph{What's in a name? That which we call a rose by any other name
    would smell as sweet.}

  William Shakespeare
\end{flushright}

\begin{flushright}
\emph{Mal nommer un objet, c'est ajouter au malheur de ce monde.}

Albert Camus
\end{flushright}

\selectlanguage{english}

\begin{abstract}
We prove that $\Out(F_N)$ is boundary amenable. This also holds more generally for $\Out(G)$, where $G$ is either a toral relatively hyperbolic group or a  finitely generated right-angled Artin group. As a consequence, all these groups satisfy the Novikov conjecture on higher signatures.
\end{abstract}

\tableofcontents

\section{Introduction}

Boundary amenability -- also known as \emph{exactness} or \emph{coarse amenability}, and also equivalent to Yu's \emph{Property A} from \cite{Yu} (as was shown by Higson and Roe in \cite{HR}) -- is a property of a countable group that has important applications in K-theory, operator algebras and measured group theory. The reader is referred to \cite{AD,Oza2} for general introductions.
 The definition is as follows.     
 
\begin{de}[Boundary amenability]\label{dfn_boundary}
  A  countable  discrete group $\Gamma$ is \emph{boundary amenable} if there exist a   nonempty compact Hausdorff space $X$ equipped with 
  an action of $\Gamma$ by homeomorphisms and a sequence of continuous maps $$\mu_n: X \to\mathrm{Prob}(\Gamma)$$ such that for all $\gamma\in \Gamma$, one has $$\sup_{x\in X}||\mu_n(\gamma.x)-\gamma.\mu_n(x)||_1\to 0$$ as $n$ goes to $+\infty$. 
\end{de}

  In this definition, $\mathrm{Prob}(\Gamma)$ denotes the space of probability measures on 
  $\Gamma$,  
  equipped with the topology of pointwise convergence, or equivalently, subspace topology from
  $\ell^1(\Gamma)$ -- the continuity of the maps $\mu_n$ in the above definition is understood with respect to this topology. An action $\Gamma\actson X$ as in Definition \ref{dfn_boundary} is called \emph{topologically amenable}.

Boundary amenability has already been established for several important classes of groups. Guentner, Higson and Weinberger proved in \cite{GHW} that all linear groups are   boundary amenable. Campbell and Niblo proved in \cite{CN} that every group acting properly and cocompactly on a   finite-dimensional CAT(0) cube complex is   boundary amenable. Boundary amenability is also known for many groups satisfying `hyperbolic-like' properties: this was established by Adams \cite{Ada} for hyperbolic groups and extended by Ozawa \cite{Oza} to the case of relatively hyperbolic groups with   boundary amenable parabolic subgroups. It was then established for mapping class groups of orientable surfaces of finite type by Kida \cite{Kid} and Hamenstädt \cite{Ham}, and for automorphism groups of locally finite buildings by Lécureux \cite{Lec}.     On the other hand, finitely generated groups whose Cayley graphs contain a properly embedded expander are not boundary amenable; examples of such groups were constructed by Gromov \cite{Gromov} (see also \cite{ArDe}),  and more recently Osajda constructed residually finite examples \cite{Osajda}. We also mention that work of Arzhantseva, Guentner and Spakula \cite{AGS} provides examples of non-coarsely amenable metric spaces of different nature. 

The goal of the present paper is to establish   the boundary amenability of $\Out(F_N)$,   the outer automorphism group of a finitely generated free group $F_N$. This group has been the subject of intensive research for a while, see e.g.\ \cite{Vog} for a recent survey.   More generally, we prove the following theorem.

\begin{theo}[see Corollaries \ref{cor-toral} and \ref{raag}]\label{main-0}
Let   a finitely generated group $G$ be either 
\begin{enumerate}
\item a free group,
\item a torsion-free Gromov hyperbolic group, 
\item a torsion-free toral relatively hyperbolic group, 
\item a right-angled Artin group.
\end{enumerate}
\noindent Then $\Out(G)$ is boundary amenable. 
\end{theo}

Since $\Aut(G)$ embeds in $\Out(G*\bbZ)$ and boundary amenability is stable under taking subgroups, boundary amenability of $\Aut(G)$ follows in these cases:

\begin{cor}
Let $G$ be a group as in Theorem~\ref{main-0}. Then $\Aut(G)$ is boundary amenable.
\end{cor}

We expect that the torsion-freeness assumption in Theorem \ref{main-0} is not necessary.
But if $G$ is a virtually torsion-free hyperbolic group, one can deduce from Theorem \ref{main-0} that $\Out(G)$ and $\Aut(G)$ are boundary amenable, see Corollary~\ref{cor_virtually_tf} (whether or not there exists a hyperbolic group which is not virtually torsion-free is a famous open question).
\\

  When $G=F_N$, an explicit compact space equipped with a topologically amenable action of $\Out(F_N)$ is constructed in Section~\ref{sec-compact} of the present paper: this space is an infinite-dimensional product space involving all boundaries of relative Outer spaces associated to free factors $A\subseteq F_N$ and free factor systems of $A$, and all boundaries of free factors of $F_N$. But our general proof strategy does not consist in working directly with this compact space; instead we rely on an inductive argument coming from Ozawa's work on boundary amenability of relatively hyperbolic groups \cite{Oza}, as will be explained later in this introduction.

  Notice that every identification between the fundamental group of a surface $\Sigma$ obtained from a closed connected surface by removing a finite non-empty set of points, and a finitely generated free group $F_N$, yields an embedding of the mapping class group of $\Sigma$ into $\Out(F_N)$. Since boundary amenability passes to subgroups, this gives a new proof of the boundary amenability of the mapping class group of every  punctured surface $\Sigma$ as above.

\paragraph*{Applications.} 
A key motivation behind the study of   boundary amenability comes from a theorem that follows from work of Yu \cite{Yu}, Higson--Roe \cite{HR} and Higson \cite{higson}, stating that   boundary amenability of $\Gamma$ implies the injectivity of the Baum--Connes assembly map, which in turn implies the Novikov conjecture on higher signatures for $\Gamma$ (this theorem builds on the fact that boundary amenability of $\Gamma$ is equivalent to $\Gamma$ satisfying Yu's property A, which implies in turn that $\Gamma$ admits a   coarse embedding in a Hilbert space). Since boundary amenability passes to subgroups \cite{Oza2}, we get the following corollary to Theorem~\ref{main-0}.

\begin{cor}\label{cor-0}
Let $G$ be a group as in Theorem~\ref{main-0}. Then $\Out(G)$ and any of its subgroups satisfy the
Novikov conjecture.
\end{cor}

Another application of the   boundary amenability of a group $\Gamma$ comes from the study of certain operator algebras associated to $\Gamma$: for example,   boundary amenability of a   countable group is equivalent to the exactness of its reduced $C^\ast$-algebra, see \cite{AD,Oza3}.

\paragraph*{Boundary amenability of the automorphism group of a free product.} In order to establish Theorem~\ref{main-0}, we actually work in the more general setting of groups   coming with a decomposition as a free product.   Let $k,N$ be non-negative integers, let $\{G_1,\dots,G_k\}$ be a finite   family of countable groups, and let $$G:=G_1\ast\dots\ast G_k\ast F_N.$$   We let $\calf=\{G_1,\dots,G_k\}$, naturally viewed as a family of subgroups of $G$. We denote by $\Out(G,\calf)$ the subgroup of $\Out(G)$ made of all automorphisms which preserve the conjugacy  class of each subgroup $G_i$, and by   $\Out(G,\calf^{(\mathrm{t})})$ the subgroup made of all automorphisms that act as the conjugation by an element $g_i\in G$ on each subgroup $G_i$. 
Our main theorem is the following.

\begin{theo}[see Theorem \ref{main-2}]\label{main}
  Let $k,N$ be non-negative integers. Let   $\calf=\{G_1,\dots,G_k\}$ be a finite   family of countable groups, and let $$G:=G_1\ast\dots\ast G_k\ast F_N.$$
\\ Assume that for all $i\in\{1,\dots,k\}$, the group $G_i$ is   boundary amenable. 
\\ Then $\mathrm{Out}(G,\calf^{(\mathrm{t})})$ is   boundary amenable.
\end{theo}

We would now like to make a few comments on the statement. First, the same statement also holds for the larger group $\Out(G,\calf)$ (instead of $\Out(G,\calf^{(\mathrm{t})})$) if we make the additional assumption that
$\Out(G_i)$ is   boundary amenable for each $i\in\{1,\dots,k\}$ (see Corollary~\ref{cor-main}). Second, we mention that (unless the given decomposition of $G$ is trivial, i.e.\ $G=G_1$), our assumption that the groups $G_i$ be   boundary amenable is necessary: indeed,   for every $i\in\{1,\dots,k\}$ the group $G_i/Z(G_i)$ embeds into $\Out(G,\calf^{(\mathrm{t})})$ as the subgroup of all partial conjugations of the $G_i$ factor, and   boundary amenability of $G_i$ is equivalent to that of $G_i/Z(G_i)$ (see Section \ref{sec-amenability}).

We briefly explain the strategy to derive Theorem~\ref{main-0} from
Theorem~\ref{main}. The particular case where $k=0$ shows that
$\Out(F_N)$ is   boundary amenable. If $G$ is torsion-free and hyperbolic or toral relatively
hyperbolic, then Theorem~\ref{main} basically reduces the proof of
Theorem~\ref{main-0} to the case where $G$ is one-ended,  in
which case JSJ theory implies that $\Out(G)$   can be, in a sense, built from
mapping class groups of surfaces and free abelian groups
 \cite{Sela_JSJ,Lev,GL4} (see Section~\ref{sec_RH}). The case where $G$ is a right-angled Artin group is proved by induction on the number of vertices of the underlying graph, using work of Charney--Vogtmann \cite{CV} (see  Section~\ref{sec_RAAG}).

\paragraph*{On the proof of Theorem~\ref{main}.}
Let us first introduce the objects that we need.
The group $\Out(G,\calf)$ has a natural action on a compact space, namely the 
compactified Outer space $\mathbb{P}\baro$  for the free product $(G,\calf)$.
Points in this space correspond to certain actions of $G$ on $\bbR$-trees \cite{Hor}.
Unfortunately, except in a few exceptional cases, the $\Out(G,\calf)$-action on $\mathbb{P}\baro$ is not topologically amenable.
Indeed, some points in $\mathbb{P}\baro$ have a non-amenable stabilizer which is a general obstruction for an action to be topologically amenable (see Corollary \ref{cor-subgroup} for instance).
For example, every tree in the boundary of Culler--Vogtmann's Outer space which is dual to a non-filling measured lamination on a surface with boundary is stabilized by every automorphism coming from a mapping class supported on the subsurface that avoids the lamination.  

We use the decomposition of $\mathbb{P}\baro=\bbP\AT\sqcup \bbP\AT^c$ into arational and non-arational trees introduced by Reynolds \cite{Rey}:  
a tree $T$ in $\bbP\baro$ is \emph{arational} if for every proper free factor $A$ of $G$ relative to $\calf$, the group $A$ is not elliptic in $T$ and the restriction of the $A$-action to its minimal subtree in $T$ is relatively free and simplicial  (see Section \ref{sec-arat-back} for more details).

\paragraph*{The factorization lemma.}
The boundary amenability of a free group $F_N$ can be obtained by showing that the action on its boundary $\partial_\infty F_N$ is topologically amenable. The key geometric feature used in the proof 
is that any two rays converging to a common point in $\partial_\infty F_N$ have the same tail.
Lemma~\ref{intro-factor} below is inspired by this phenomenon and is crucial in our proof of Theorem~\ref{main}.

Let $T$ be an arational $(G,\calf)$-tree, and let $S$ be a simplicial $(G,\calf)$-tree coming with a morphism $f:S\to T$ (see Section~\ref{sec-back1} for definitions). We define the \emph{turning class} of $f$ as the set of turns (i.e.\ pairs of directions) at branch points in $T$ which lift to $S$.  Our key observation is the following.

\begin{lemma}[Factorization lemma]\label{intro-factor}
Let $T\in\AT$, let $S,S'\in\calo$, and let $f:S\to T$ and $f':S'\to T$ be two optimal morphisms having the same turning class.
\\ Then there exists $\epsilon>0$ such that if $U\in\calo$ is a tree of covolume at most $\epsilon$ and $f$ factors through $U$, then $f'$ also factors through $U$.
\end{lemma}

In fact, the factorization lemma holds more generally under the assumption that $T$ has trivial arc stabilizers (see Lemma \ref{lem_blow_up}).

 The idea of  turning classes and its consequences on factorization appear in an unpublished paper by Los--Lustig \cite{LL} (where these were called \emph{blowup classes}).

This factorization lemma has an interpretation that parallels the case of the free group mentioned above.
Indeed, given an optimal morphism $f:S\ra T$,
each folding path from $S$ to $T$ guided by $f$ is a geodesic in Outer space for the asymmetric Lipschitz metric \cite{FM}.
The fact that $f$ and $f'$ factor through the same trees of small enough volume
means that the union of tails of all such geodesic paths associated to $f$ and $f'$ coincide.

It is noticeable that here, factorization occurs ``on the nose'', while in the proofs of boundary amenability of hyperbolic groups \cite{Oza} or mapping class groups \cite{Ham,Kid}, one appeals to averaging arguments on finite collections of rays going to the same boundary point (in the case of mapping class groups, Kida uses \emph{tight geodesics} to get this finiteness).\footnote{Here we mention that after the first version of our paper was released, new proofs of the boundary amenability of hyperbolic groups and mapping class groups were obtained by Marquis--Sabok \cite{MS} and Przytycki--Sabok \cite{PS}, yielding in each case the stronger conclusion of hyperfiniteness of a certain boundary action; in these proofs some form of  factorization ``on the nose'' is established.}

We also prove a ``finite width'' statement about the collection of rays going to a given  arational tree. This is given by the following important fact, established in Section~\ref{sec-finite-width}: given $t\in\mathbb{R}$, the set $\cals_t(f)$ of all simplices
of relative Outer space that contain a tree of covolume $e^{-t}$ through which the morphism $f$ factors is finite.

\paragraph*{Amenability on sets of arational actions.}
Denoting by $\Simp$ the (countable) set of simplices of relative Outer space, we define probability measures on $\Simp$
as follows.
Given $t\in\mathbb{R}$, we define $\nu_t(f)$ to be the uniform probability measure on the finite set $\cals_t(f)$, and we then average
it on long segments by defining  $$\mu_n(f):=\frac{1}{n}\int_{n}^{2n}\nu_t(f)dt.$$  
The factorization lemma shows that as long as $f:S\ra T$ and $f':S'\ra T$ have the same turning class,  we have $\nu_t(f)=\nu_t(f')$ for all sufficiently large $t\in \bbR$.
If now $T$ and $T'$ are in the same projective class, and $f:S\ra T$ and $f':S'\ra T'$ 
have the same turning class, then
 $\mu_n(f)$ and $\mu_n(f')$ get arbitrarily close as $n$ goes to infinity.

This is enough to deduce the Borel amenability of the action of $\Out(F_N)$ on
the set $\bbP\ATf\subset \bbP\baro$ of \emph{free} arational $F_N$-trees.
In contrast with topological amenability,
Borel amenability asks for a sequence of maps $\mu_n:\bbP\ATf\ra \Prob(\Out(F_N))$ that are asymptotically equivariant as in Definition \ref{dfn_boundary}, but that are only required to be measurable (see Definition \ref{Borel-def}).

\begin{theo}[see Theorem \ref{thm:FN-case}]\label{thm:FN-case-intro} 
  The action of $\Out(F_N)$ on $\bbP\ATf$ is Borel amenable.
\end{theo}

This result is significantly easier than the general case (Theorem \ref{arat} below)
which is needed even for the proof of boundary amenability of $\Out(F_N)$.
The main simplification is that freeness of $T$ implies that there are finitely many orbits of turns
at each branch point, so the number of turning classes of morphisms targeting $T$ is finite.
We can therefore define a probability measure $\mu_n(T)$  by averaging the probability measures $\mu_n(f)$
for a finite set of morphisms $f$ representing all possible turning classes.
Since stabilizers of simplices of the Culler--Vogtmann Outer space of a free group are finite, it is easy
to transfer these probability measures to $\Out(F_N)$ so one deduces Borel amenability of the action.

\paragraph*{The general case and the inductive argument.}
The key statement that we prove in general is the following. 

\begin{theo}[see Theorem \ref{arat-2}]\label{arat}
  Let $k,N$ be non-negative integers. Let $\calf=\{G_1,\dots,G_k\}$ be a finite family of countable groups, and let $$G=G_1\ast\dots\ast G_k\ast F_N.$$ Then there exists a sequence of Borel maps $$\mu_n:\mathbb{P}\AT\to\mathrm{Prob}(\Simp)$$ 
such that for all $\Phi\in\Out(G,\calf)$ and all $T\in\mathbb{P}\AT$, one has $$||\Phi.\mu_n(T)-\mu_n(\Phi.T)||_1\to 0$$ as $n$ goes to $+\infty$.
\end{theo}

Using an additional argument, we actually prove that the
$\Out(G,\calf^{(\mathrm{t})})$-action on $\AT$ is Borel amenable
under the (necessary) additional assumption that  
 for each $i\leq k$, nontrivial elements of $G_i$ 
have amenable centralizers (see Theorem \ref{at-amen}).
\\

Before commenting on the proof, let us first explain how to deduce boundary amenability of $\Out(G,\calf^{(\mathrm{t})})$ (Theorem \ref{main}).
Note that the stabilizers of elements of $\Simp$ 
are usually infinite.
In addition, the set $\bbP\AT$ of projective arational trees is not compact.
We therefore rely on an inductive argument inspired from Kida's proof of   boundary amenability of mapping class groups \cite{Kid}.
Indeed, $\bbP\AT$ is contained in the compact space $\bbP\baro$;
moreover, to any tree in $\bbP\baro\setminus\bbP\AT$, one can canonically associate a   nonempty finite set of   conjugacy classes of   relative free factors \cite{Rey,Hor2}.
Then using a theorem by Ozawa \cite{Oza} (see Corollary \ref{ozawa-kida}),
it suffices to prove boundary amenability
of the stabilizer of each relative free factor, and of the stabilizer of each element of $\Simp$.
This allows to argue by induction on the complexity of the free product decomposition.

The base cases of the induction correspond to either 
\begin{itemize}
\item $G=G_1\ast G_2$ and $\calf=\{G_1,G_2\}$, in which case $\Out(G,\calf^{(\mathrm{t})})$ is isomorphic to $G_1/Z(G_1)\times G_2/Z(G_2)$, or 
\item $G=G_1\ast\mathbb{Z}$ and $\calf=\{G_1\}$, in which case $\Out(G,\calf^{(\mathrm{t})})$ has an index $2$ subgroup isomorphic to $(G_1\times G_1)/Z(G_1)$,   where $Z(G_1)$ is diagonally embedded in $G_1\times G_1$. 
\end{itemize}

\paragraph*{Names and proof of Theorem \ref{arat}.}
In comparison to Theorem \ref{thm:FN-case-intro}, the main difficulty in the proof
comes from the fact that given $T\in\AT$, the set of turning classes of morphisms targeting $T$
may be infinite, so one cannot average over them any more.

The rough idea to bypass this difficulty is to enumerate the possible turning classes in $T$ by giving them ``names'' in an $\Out(G,\calf)$-invariant way. Only finitely many turning classes should share a given name.
We then define $\mu_n(T)$ by averaging $\mu_n(f)$ over a finite set of morphisms $f$ whose turning classes coincide with the set of turning classes having the first possible name.

Our enumeration of turning classes is made differently for geometric and nongeometric trees in $\AT$. If $T\in\AT$ is nongeometric (in the sense of Levitt--Paulin \cite{LP}), then it can be strongly approximated by trees in $\calo$. It follows that there exists a morphism $f:S\to T$ with $S\in\calo$, such that every turn in $T$ lifts to $S$. The turning class of $f$ (which contains all turns) is then our preferred turning class, i.e.\ the first from the enumeration.

When $T$ is geometric, the band complex that resolves $T$ enables us to analyze 
a very particular set of turns that we call \emph{ubiquitous}.
These are defined as turns $(d,d')$ such that for every nondegenerate segment $I\subseteq T$, there exists $g\in G$ such that $gI$ contains $(d,d')$.
We show in particular that each direction is contained in only finitely many ubiquitous turns. 
These turns can therefore be used to 
define an \emph{angle} between two directions $d,d'$ at $x$ by counting the number of overlapping ubiquitous turns at $x$ needed to go from $d$ to $d'$ (see Definition \ref{def_angle}).
Using a Whitehead graph argument, we also prove that there are sufficiently many turns so that the angle between any pair of directions coming from the same minimal component of the band complex
is finite.
The ``name'' of a turning class is then defined from the angles of the turns it contains. 

The above discussion is made formal in the following technical statement, which is key in our proof of Theorem~\ref{arat}. Here $\turn{\AT}$ is the set of all turning classes on arational trees (this space is equipped with a $\sigma$-algebra). The symbol $\perp$ is a special name for turning classes that we want to ignore.

\begin{prop}[see Proposition \ref{prop_namable}]\label{name-intro}
There exists a measurable map $$\Name:\turn{\AT}\to\mathbb{N}\cup\{\perp\}$$ such that
\begin{itemize}
\item if $\calb$ is a turning class on a tree $T$, and $\lambda\calb$ is the turning class on the tree $\lambda T$ for some $\lambda>0$ containing the same turns as $\calb$, then $\Name(\lambda\calb)=\Name(\calb)$,
\item $\Name$ is $\Out(G,\calf)$-invariant, i.e.\ $\Name(\Phi.\calb)=\Name(\calb)$ for all $\calb\in\turn{\AT}$ and all $\Phi\in\Out(G,\calf)$,
\item for all $T\in \AT$ and all $n\in\mathbb{N}$, there are only finitely many turning classes on $T$ whose name is $n$, and
\item for all $T\in \AT$, there exists an optimal morphism with range $T$ whose turning class with respect to $T$ has a name different from $\perp$.
\end{itemize}          
\end{prop}

We finish by mentioning that Theorem~\ref{main} does not yield a concrete compact space equipped with a topologically amenable action of $\Out(G,\calf^{(\mathrm{t})})$. 
However, in the case where $G$ is a free group, we describe an explicit such space
by unraveling our inductive argument. This is a metrizable compact space obtained as a product space involving in particular all boundaries of the relative Outer spaces associated to all free factors $A\subseteq F_N$ and all free factor systems of $A$, and all Gromov boundaries of free factors of $F_N$ (see Section \ref{sec-compact}).  

\paragraph*{Structure of the paper.} The paper is organized as follows. In Section~\ref{sec-back}, we review necessary background on free products and their Outer spaces, arational trees, geometric and nongeometric trees, and general facts concerning boundary amenability of groups. 
 
In Section~\ref{sec-scheme}, we prove the factorization lemma (Lemma~\ref{intro-factor}). We then define the  probability measures associated to any arational tree,
 and deduce the amenability of the action of $\Out(F_N)$ on the set of free arational $F_N$-trees (Theorem \ref{thm:FN-case-intro}).
We then explain  how this argument extends to derive Theorem~\ref{arat} from Proposition~\ref{name-intro}.

In Section~\ref{sec-name}, we give names to turning classes on arational trees to prove Proposition~\ref{name-intro}; we treat nongeometric and geometric trees separately.

In Section~\ref{sec-appl}, we complete the proof of our main theorem (Theorem~\ref{main}), and we explain how to derive boundary amenability of $\Out(G)$ when $G$ is either a toral relatively hyperbolic group or a right-angled Artin group.

In Section~\ref{sec-complements}, we establish two complementary results to our main theorem. First, we unravel the inductive argument to build an explicit compact metrizable space equipped with a topologically amenable action of $\Out(F_N)$. Second, we prove that if the centralizer of every nontrivial peripheral element is amenable, then the $\Out(G,\calf^{(\mathrm{t})})$-action on $\AT$ is Borel amenable.  

\paragraph*{Acknowledgments.} The second and third authors would like
to thank Martin Lustig for explaining to us how turning classes can be used to get factorization results during a conference held in Carcassonne in Fall 2014, supported by the ANR LAM. All three authors were supported by the National Science Foundation under
Grant No. DMS-1440140 while the authors were in residence at the
Mathematical Sciences Research Institute in Berkeley, California,
during the Fall 2016 semester. The first author was supported by the
NSF under Grant No. DMS-1607236.
The second author acknowledges support from the Institut Universitaire de France and from the Centre Henri Lebesgue ANR-11-LABX-0020 LEBESGUE. 
The third author acknowledges support from the Agence Nationale de la Recherche under Grant ANR-16-CE40-0006.

\section{Background}\label{sec-back}

\subsection{General background on free products}\label{sec-back1}

\paragraph*{Generalities on free products.} Let $\calf:=\{G_1,\dots,G_k\}$ be a finite collection of countable groups, and let $$G:=G_1\ast\dots\ast G_k\ast F_N,$$ where $F_N$ is a free group of rank $N$. Elements or subgroups of $G$ that are conjugate into one of the subgroups in $\calf$ are called \emph{peripheral}. We define the \emph{complexity} of $(G,\calf)$ as the pair $\xi(G,\calf):=(N+k-1,N)$.   Complexities are ordered lexicographically. 
We say that $(G,\calf)$ is \emph{sporadic} if $\xi(G,\calf)\leq (1,1)$ which happens exactly if $G=G_1$, $G=\bbZ$, $G=G_1*G_2$, or $G=G_1*\bbZ$.

A \emph{$(G,\calf)$-tree} is an $\mathbb{R}$-tree $T$ equipped with a $G$-action   by isometries, such that every subgroup in $\calf$ acts elliptically on $T$ (i.e.\ it fixes a point in $T$). A \emph{$(G,\calf)$-free splitting} is a simplicial $(G,\calf)$-tree with trivial edge stabilizers. A \emph{$(G,\calf)$-free factor} is a subgroup of $G$ that arises as a point stabilizer in some $(G,\calf)$-free splitting. A \emph{proper} $(G,\calf)$-free factor is a free factor distinct from $G$ and non-peripheral.
A \emph{$(G,\calf)$-free factor system} is the collection of all conjugacy classes of point stabilizers in some $(G,\calf)$-free splitting.

\newcommand{\bigast}{\mathop{\scalebox{1.5}{\raisebox{-0.2ex}{$\ast$}}}}
A theorem of Kurosh \cite{Kur} asserts that every subgroup $A\subseteq G$ decomposes as  $(\bigast_j H_j)\ast F$,
where $F$ is a free group and
a subgroup of $A$ is peripheral in $(G,\calf)$ if and only if it is $A$-conjugate into some $H_j$. 
We denote by $\calf_{|A}$ the collection of all $A$-conjugacy classes of the subgroups $H_j$ from the above decomposition of $A$.

\paragraph*{Directions and branch points in trees.} 

A \emph{direction} $d$ at $x$ in a $(G,\calf)$-tree $T$ is a germ of isometric maps $\eta:[0,\eps]_{\bbR}\ra T$ with $\eta(0)=x$.
We say that a subtree $Y\subset T$ \emph{contains} the direction $d$ if it is represented by a map $\eta$ whose image is contained in $Y$.

A \emph{turn} at a point $x\in T$ is a pair $(d,d')$ where $d,d'$ are two distinct directions at $x$ (a turn is formally defined   as an ordered pair, but the ordering will not be important).
The point $x$ is an \emph{inversion} point if there are exactly two directions at $x$, 
and there exists $g\in G$ fixing $x$ and swapping the two directions at $x$.
A \emph{branch point} $x\in T$ is a point such that $T\setminus \{x\}$ has at least 3 connected components. A point that is either a branch point or an inversion point is called a \emph{generalized branch point}.
Note that in this paper, 
inversion points will occur only if one of the factors in $\calf$ is isomorphic to $\bbZ/2\bbZ$. A \emph{branch direction} in $T$ is a direction based at a branch point.

\paragraph*{Outer space and its closure.}

A $(G,\calf)$-tree $T$ is \emph{relatively free} if every element of $G$ elliptic in $T$ is conjugate into one of the subgroups in $\calf$. A \emph{Grushko $(G,\calf)$-tree} is a simplicial metric relatively free $(G,\calf)$-tree. 
We note that $(G,\calf)$ is non-sporadic if and only if every Grushko tree has at least 2 orbits of edges.
The \emph{unprojectivized relative Outer space} $\calo$ is the space of all $G$-equivariant isometry classes of Grushko $(G,\calf)$-trees   \cite{CuV,GL}. We denote by $\mathbb{P}\calo$ the projectivized Outer space, where trees are considered up to homothety instead of isometry. Given a tree $S\in\mathcal{O}$, the set of all trees $S'$ obtained by keeping the same underlying simplicial structure but varying the metrics (keeping all edge lengths positive) projects to an open simplex in $\mathbb{P}\calo$. We denote by $\Simp$ the countable collection of all these open simplices. Given an open simplex $\Delta\in\Simp$, we denote by $\Tilde{\Delta}$ the preimage of $\Delta$ in $\calo$.

The closure $\baro$ of $\calo$ in the space of all $(G,\calf)$-trees (equipped with the Gromov--Hausdorff topology introduced in \cite{Pau}) was identified in \cite{Hor} with the space of all \emph{very small} $(G,\calf)$-trees, i.e.\ those trees $T$ for which tripod stabilizers are trivial, and arc stabilizers are either trivial, or maximally cyclic and nonperipheral (see \cite{CL,BF} for free groups).
 We will let $\partial\calo:=\baro\setminus\calo$.

\paragraph*{Levitt decomposition of a very small tree.}

Very small trees  with dense $G$-orbits have trivial arc stabilizers, see \cite[Proposition~4.17]{Hor}. By \cite{Lev} (see also \cite[Theorem~4.16]{Hor} for free products),
every tree $T\in\baro$ splits in a unique way as a graph of actions
(in the sense of \cite{Lev}), such that vertices of this decomposition correspond to orbits of 
 connected components of the closure
of the set of generalized branch points. Edges of this decomposition correspond to orbits of maximal arcs whose interiors contain no generalized branch points.
In particular, vertex groups act with dense orbits on the corresponding subtree (maybe a point).
The Bass--Serre tree of the underlying graph of groups is very small (maybe trivial). 
This decomposition is called the \emph{Levitt
  decomposition} of $T$.

\paragraph*{Morphisms.} Given two trees $S,T\in\baro$, a \emph{morphism} $f:S\to T$ is a $G$-equivariant map such that every segment in $S$ can be subdivided into finitely many subsegments, in such a way that $f$ is an isometry when restricted to any of these subsegments. We denote by $\text{Mor}$ the space of all morphisms between trees in $\baro$, equipped with the Gromov--Hausdorff topology (see \cite[Section 3.2]{GL} for a definition of the Gromov--Hausdorff topology on the set of morphisms). A morphism is \emph{optimal} if every point $x\in S$ is contained in the interior of a nondegenerate segment $I_x$ such that $f_{|I_x}$ is injective. We denote by $\Opt$ the subspace of $\Mor$ consisting of all optimal morphisms. Given subsets $\sigma\subseteq\calo$ and $\tau\subseteq\baro$, we will denote respectively by $\Opt_\sigma$, $\Opt_{\ra \tau}$ and $\Opt_{\sigma\ra \tau}$ 
the set of optimal morphisms whose source tree lies in $\sigma$, or whose range tree lies in $\tau$, or both.
Given $T\in \baro$, we will also use the notation $\Opt_{\ra T}$ for the set of optimal morphisms whose range tree is $T$.

\begin{lemma}\label{opt-bor}
Let $\sigma\subseteq\calo$ and $\tau\subseteq\baro$ be Borel subsets. Then $\Opt_{\sigma\to\tau}$ is a Borel subset of $\Mor$.
\end{lemma}

\begin{proof}
Since the source and range maps from $\Mor$ to $\baro$ are continuous, it is enough to check that $\Opt$ is a Borel subset of $\Mor$. Since $\Simp$ is countable, it is enough to show that for all $\Delta\in\Simp$, letting $F_\Delta\subseteq\calo$ be the closure of the subset made of all trees projecting to $\Delta$, the space $\Opt_{F_\Delta}$ is a closed subset of $\Mor$. Let $S\in F_\Delta$, let $T\in\baro$, let $f:S\to T$ be a morphism, and let $(f_n)_{n\in\mathbb{N}}\in\Opt^\mathbb{N}$ be a sequence of optimal morphisms converging to $f$, with sources in $F_\Delta$; we will prove that $f$ is optimal. Assume not. Then there exists a vertex $v\in S$ such that any two segments $[v,v_1]$ and $[v,v_2]$ in $S$ with extremity $v$ have initial subsegments that are identified by $f$. Since arc stabilizers in $T$ are  either trivial or nonperipheral, this implies that the stabilizer $G_v$ of $v$ is trivial because $S\in \calo$. So $v$ has finite valence. 
In particular, there exists $\epsilon>0$ such that we can find subsegments $[v,x_i]$ of length $\epsilon$ on each of the edges $e_i$ incident on $v$, which are all identified by $f$. For all $n\in\mathbb{N}$, the source $S_n$ of $f_n$ is in $F_\Delta$. Up to passing to a subsequence, we can assume that all trees $S_n$ have the same underlying simplicial tree, and one passes from $S_n$ to $S$ by slightly changing the edge lengths and collapsing a $G$-invariant forest. Let $F_n$ be the preimage of $v$ in $S_n$ under the collapse map. The edges $e_i^n$ incident on $F_n$ in $S_n$ are in natural bijection with the edges $e_i$ in $S$. For all $n\in\mathbb{N}$ sufficiently large, any approximation $x_i^n$ of the point $x_i$ in $S_n$ given by the equivariant Gromov--Hausdorff topology lies on the edge $e_i^n$, any approximation $v_n$ of $v$ lies at distance at most $\epsilon/100$ from the forest   $F_n$, and all images $f_n(x_i^n)$ lie in the same connected component of $T\setminus\{f_n(v_n)\}$. This implies that there exists a point $y_n$ in the $\epsilon/100$-neighborhood of $F_n$ such that all directions at $y_n$ in $S_n$ are mapped to the same direction in $T$ by $f_n$. In other words $f_n$ is non-optimal, a contradiction.  
\end{proof}

\subsection{Arational trees}\label{sec-arat-back}

A tree $T\in\partial\calo$ is \emph{arational} if for every proper $(G,\calf)$-free factor $A$,   the group $A$ is not elliptic in $T$, and the $A$-action on its minimal subtree $T_A$ is equivariantly isometric to a Grushko $(A,\calf_{|A})$-tree (\cite{Rey,Hor2}). We denote by $\AT$ the subspace of $\baro$ made of arational trees; we mention that it is a Borel subset of $\baro$, see \cite[Lemma~5.5]{Hor2}. Notice that all arational trees have dense $G$-orbits.
Indeed, a tree $T$ without dense orbits has a nontrivial Levitt decomposition $\Lambda$. 
Assume first that all edge stabilizers of $\Lambda$ are trivial.
Then if all the vertex stabilizers of the Levitt decomposition are peripheral, then $T$ is a Grushko tree, contrary to the assumption $T\in \partial\calo$.
Otherwise, some vertex stabilizer of the Levitt decomposition is a proper free factor, and its action on its minimal subtree (maybe a point)
has dense orbits, so $T$ is not arational.
Finally, if some edge stabilizer of $\Lambda$ is non-trivial,
 then some edge stabilizer of $\Lambda$ is contained in a proper $(G,\calf)$-free factor by \cite[Lemma~5.11]{Hor}, so $T$ is not arational.

\paragraph*{Arational surface trees.} 

One can construct examples of arational $(G,\calf)$-trees coming from trees that are dual to arational laminations on certain $2$-orbifolds. In the context of free products, arational surface trees were introduced in \cite[Section 4.1]{Hor2}, and are defined in the following way (see Figure~\ref{fig-arat-surf} for an illustration of the construction).

\begin{figure}
\begin{center}
\includegraphics[width=.5\textwidth]{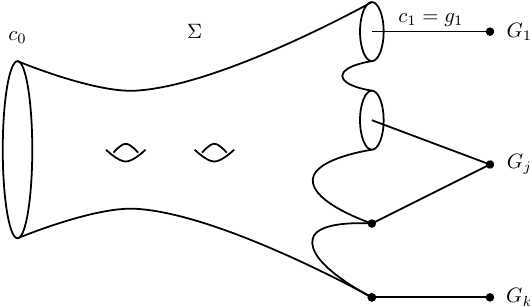}
\caption{The underlying graph of groups of an arational surface tree.}
\label{fig-arat-surf}
\end{center}
\end{figure}

Let $\Sigma$ be a $2$-orbifold with only conical singularities, having $s+1$ boundary components $c_0,\dots,c_s$ and $q$ conical points $c_{s+1},\dots,c_{s+q}$. Let $\Gamma$ be a graph of groups having a \emph{central} vertex with vertex group isomorphic to $\pi_1(\Sigma)$, and $k$ other vertices with vertex groups isomorphic to the peripheral subgroups $G_i$. For all $j\in\{1,\dots,s+q\}$, we choose a peripheral subgroup $G_{i_j}$, and we amalgamate it with the cyclic subgroup of $\pi_1(\Sigma)$ generated by $c_j$, identifying $c_j$ with an element of $G_{i_j}$ having the same order
 (notice that we allow the case where $c_j$ generates $G_{i_j}$, in this case the Bass--Serre tree of $\Gamma$ may fail to be minimal). The boundary curve $c_0$ is left unused. Choices are made so that the graph $\Gamma$ is connected, i.e.\ every peripheral subgroup is joined to the central vertex by an edge. We assume that the fundamental group of $\Gamma$ is isomorphic to $G$, and we choose such an isomorphism. We now build an action of $G$ on an $\bbR$-tree from 
a \emph{graph of actions} on $\Gamma$ as defined in \cite{Lev1} (except that edges are given length $0$): choose a $\pi_1(\Sigma)$-action on an $\bbR$-tree  dual to an arational lamination on $\Sigma$;
then take a copy $Y_v$ of this tree for each vertex $v$ of the Bass--Serre tree $S_\Gamma$ of $\Gamma$  that projects to the central vertex of $\Gamma$, and take a point $x_u$ for 
each vertex $u$ of $S_\Gamma$  that does not project to the central vertex; then for each edge $uv$ of $S_\Gamma$, glue the point $x_u$ to the unique point in $Y_v$
that is fixed by the stabilizer of $uv$. Every tree obtained by this construction is called an \emph{arational surface} tree.

It was proved in \cite[Section 4.1]{Hor2} that arational surface trees are indeed arational. In addition, every arational $(G,\calf)$-tree is either relatively free or arational surface: this was first proved by Reynolds \cite{Rey} for free groups, and extended to free products in \cite[Lemma~4.6]{Hor2}.

\paragraph*{Maps towards arational trees.}

Given $T,T'\in\baro$, the \emph{bounded backtracking constant} $\BBT(f)$ of a Lipschitz map $f:T\ra T'$ is the supremum of $d_{T'}(f(y),[f(x),f(z)])$
over all points $x,y,z\in T$ aligned in this order.
By \cite[Lemma 3.1]{BFH}, \cite[Proposition 3.12]{Hor3}, if $T'$ is very small and $T$ is a Grushko tree, then for every Lipschitz map $f:T\to T'$, we have  $\BBT(f)\leq \mathrm{Lip}(f)\covol(T)$. The following fact extends \cite[Corollary~3.9]{Hor1} to the case of free products.

\begin{lemma}\label{morphism-arat-0}
Let $T,T'\in\baro$, and assume that $T'$ has dense orbits.
\\ Then every $1$-Lipschitz map $f:T'\to T$ preserves alignment.
\end{lemma}

\begin{proof}
Assume by contradiction that it does not: there exist $x,y,z\in T'$ aligned in this order so that $f(y)$ is at positive distance (which we denote by $\epsilon$) from $[f(x),f(z)]$. Since $T'$ has dense orbits, it has trivial arc stabilizers, and therefore we can find a sequence of trees $S_n\in\calo$ converging to $T'$ and coming with $1$-Lipschitz morphisms $f_n:S_n\to T'$, see \cite[Theorem~5.3]{Hor}. By postcomposing $f_n$ by the map $f:T'\to T$, we get $1$-Lipschitz maps $g_n:S_n\to T$. The covolumes of the trees $S_n$ converge to $0$; since $\BBT(g_n)\le \covol(S_n)$, we deduce that $\BBT(g_n)$ converges to $0$. For all $n\in\mathbb{N}$, let $x_n$ (resp. $z_n$) be an $f_n$-preimage of $x$ (resp. $z$) in $S_n$. Then there exists $y_n\in [x_n,z_n]$ such that $f_n(y_n)=y$. Now we have $d_{T}(g_n(y_n),[g_n(x_n),g_n(z_n)])=d_{T}(f(y),[f(x),f(z)])=\epsilon$. Therefore $\BBT(g_n)\ge \epsilon$, which contradicts the fact that $\BBT(g_n)$ converges to $0$. 
\end{proof}

\begin{lemma}\label{morphism-arat-1}
Let $S\in\baro$ be a tree without dense orbits, let $T\in\AT$. If there exists a $1$-Lipschitz map from $S$ to $T$, then $S\in\calo$.
\end{lemma}

\begin{proof}
Let $\Lambda$ be the  Levitt decomposition of $S$. If $\Lambda$ contains an edge with nontrivial (whence nonperipheral) stabilizer, then there exists an edge $e\subseteq\Lambda$ whose stabilizer is a cyclic subgroup $G_e$ which is contained in a proper $(G,\calf)$-free factor $A$. Then $G_e$ is elliptic in $T$, so the minimal $A$-tree in $T$ is not Grushko, contradicting arationality of $T$. Therefore $\Lambda$ only has edges with trivial stabilizer. If $\Lambda$ has a nonperipheral vertex stabilizer $G_v$, then $G_v$ is a proper $(G,\calf)$-free factor and acts with dense orbits on its minimal subtree in $S$ (which can be a point). Since $f$ is $1$-Lipschitz, the free factor $G_v$ acts with dense orbits on its minimal subtree in $T$ (which can be a point), again contradicting arationality of $T$. Therefore all vertex stabilizers of $\Lambda$ are peripheral, in other words $S$ is a Grushko tree.
\end{proof}

\begin{cor}\label{morphism-arat}
Let $T\in\AT$, and let $T'\in\partial\Os$. If $T'\neq T$, then there is no morphism from $T'$ to $T$.
\end{cor}

\begin{proof}
Assume that there exists a morphism $f:T'\to T$. Lemma~\ref{morphism-arat-1} implies that $T'$ has dense orbits, and Lemma~\ref{morphism-arat-0} then implies that $f$ preserves alignment. Since every alignment-preserving morphism is an isometry, we have $T'=T$.  
\end{proof}

\subsection{Geometric and nongeometric trees}\label{sec-geom-back}

\paragraph*{Definition of geometric trees.} Geometric trees were first defined by Levitt--Paulin \cite{LP}, see also \cite{Hor} for the case of free products. Fix $R\in\calo$ a Grushko tree. Let $T\in\baro$.   Let $\calk=(K_v)_{v\in R}$ be an equivariant family of subtrees $K_v\subset T$ (called \emph{base trees}) indexed by vertices $v\in R$, such that for each vertex $v$, we have $K_v=G_v.K'$ for some finite subtree $K'\subset K_v$. 
 We assume that for each edge $vv'$ of $R$, the intersection $K_v\cap K_{v'}$ is non-empty. We construct a band complex $\Sigma_{\calk}(T,R)$ in the following way. Starting from the disjoint union of the base trees, the foliated band complex $\Sigma_\calk(T,R)$ is obtained 
by gluing for each edge $vv'$ of $R$ 
a band foliated by vertical segments, joining the copy of $K_v\cap K_{v'}$ in $K_v$ to its copy in $K_{v'}$. 
The space of leaves of 
$\Sigma_\calk(T,R)$ has a natural structure of an $\mathbb{R}$-tree $T_\calk$, see \cite{LP}.  
Notice that since $\Sigma_\calk(T,R)$ is modeled on the Grushko tree $R$, all leaves of $\Sigma_\calk(T,R)$ are trees. A tree $T\in\baro$ is \emph{geometric} if there exists a family of finite subtrees $\calk$ as above such that $T=T_\calk$.

The quotient $\Sigma_\calk(T,R)/G$ is a compact foliated 2-complex. Indeed, each $K_v/G_v$ is a finite tree by assumption,
there are only finitely many orbits of bands in  $\Sigma_\calk(T,R)$, and each band $B$ in $\Sigma_\calk(T,R)$ 
is a product $K_B\times [0,1]$ of a finite tree by an interval,
and $K_B\times (0,1)$ embeds in the quotient because edges of $R$ have trivial stabilizer 
(the quotient map may however fold $K_B\times \{0\}$ or $K_B\times\{1\}$).
One can then associate to this foliated $2$-complex a system of partial isometries on a finite tree (see \cite[Section~3.3]{Hor}). This will allow us to apply standard results such as \cite{GLP}.

\paragraph*{Characterization using strong limits.} Geometric trees can be characterized as those that do not occur as nonstationary strong limits of trees in $\baro$. Precisely, a sequence of trees $S_n\in\baro$ is a \emph{direct system} if it comes with a collection of morphisms $f_{nn'}:S_n\ra S_{n'}$ (for all $n\le n'$)
 and morphisms $f_n:S_n\ra T$ (for all $n\in\mathbb{N}$) making the following diagram commute: 
$$\xymatrix@C=2cm@R=1cm{
&&&T\\
S_n\ar@/_1pc/[rr]_{f_{nn''}} \ar[r]|{f_{nn'}} \ar[rrru]^{f_n}
& S_{n'} \ar[r]|{f_{n'n''}} \ar[rru]|{f_{n'}} &S_{n''} \ar[ru]_{f_{n''}}&
}
$$
The direct system $(S_n)_{n\in\mathbb{N}}$ \emph{converges strongly} to $T$ if for every finite subtree $X\subseteq S_n$, there exists $n'>n$ such that $f_{n'}$ is  isometric when restricted to $f_{nn'}(X)$.

It was proved in \cite{LP} that a tree $T\in\baro$ is nongeometric if and only if it is a nonstationary strong limit of a direct system of trees in $\baro$ 
 (this is proved for finitely presented groups in \cite{LP}, but this generalizes without difficulty
to any finite free product when the factors are elliptic).
If additionally $T$ is arational, then Corollary~\ref{morphism-arat} enables us to add the requirement that $S_n\in \Os$, as opposed to being in the boundary of Outer space. We record this in form of the following lemma.

\begin{lemma}\label{strong-cv}
For all $T\in\AT$, the following statements are equivalent.
\begin{itemize}
\item The tree $T$ is nongeometric.
\item There exists a nonstationary direct system of trees $(S_n)_{n\in\mathbb N}\in\Os^{\mathbb{N}}$ that converges strongly to $T$.
\item For every finite subset $F\subseteq G$, there exists $S\in\calo$  with a morphism $S\ra T$ such that $||g||_{S}=||g||_T$ for all $g\in F$.
\end{itemize}
\end{lemma}

\begin{proof}
  The equivalence between the first two statements follows from \cite{LP} as mentioned above.
The second assertion easily implies the third.

Assume that the third assertion holds. 
We are first going to construct trees with morphisms $S_1\ra S_2\ra S_3\dots$ and morphisms
$f_j:S_j\ra T$ without any commutation requirement.
We choose an exhaustion of $G$ by finite subsets $B_1\subset B_2\subset\dots$.
We start with $S_1\in \calo$ with any morphism $f_1:S_1\ra T$.
Assuming that $S_i$ has already been constructed, consider $F_i$ a finite subset of $G$
containing $B_i$ and a finite set of \emph{candidates} for $S_i$; 
this means that there is a morphism $S_i\ra S'$ if and only if $||g||_{S_i}\geq ||g||_{S'}$ for all $g\in F_i$.
By assumption, there exist $S_{i+1}$ with a morphism $f_{i+1}:S_{i+1}\ra T$ such that  $||g||_{S_{i+1}}= ||g||_{T}$ for all $g\in F_i$.
Since $||g||_{S_i}\geq  ||g||_{T} = ||g||_{S_{i+1}} $ for all $g\in F_i$, there exists a morphism $S_i\ra S_{i+1}$, and our inductive construction is complete.

In order to make the diagram commutative,  define $f_i^j:S_i\ra T$ for $i\leq j$ as the composition of the map $S_i\ra S_j$ with $f_j:S_j\ra T$.
For every $j$, the following diagram commutes by construction:
$$\xymatrix{
S_1\ar[r]  \ar[rrrd]_{f_1^j}
& S_{2} \ar[r] \ar[rrd]|{f_{2}^j} &  \cdots  \ar[r] &S_{j} \ar[d]^{f_{j}^j}\\
&&&T 
}
$$
By diagonal extraction, one can find a subsequence such that 
for every $i$, $f_i^j$ converges to $f_i^\infty$ (in the topology on the set of morphisms) and the following infinite diagram commutes:
$$\xymatrix{
S_1\ar[r]  \ar[rrrd]_{f_1^\infty}
& S_{2} \ar[r] \ar[rrd]|{f_{2}^\infty} &  \cdots  \ar[r] &S_{i}\ar[r] \ar[d]^{f_i^\infty}&\cdots\\
&&&T 
}
$$
This gives us a direct system of trees $(S_i)_{i\in\bbN}$ with morphisms to $T$ such that for every finite subset $F\subset G$,
there exists $i$ such that $||g||_{S_i}=||g||_{T}$ for all $g\in F$. It easily follows that $S_i$ converges strongly to $T$.
\end{proof}

\paragraph*{Skeleton of a geometric tree.} 
Recall that a \emph{transverse family} $\mathcal{Y}$ in a tree $T\in\baro$ is a $G$-invariant collection of nondegenerate subtrees such that any two distinct subtrees in the collection intersect in at most one point. A tree $T\in\baro$ is \emph{indecomposable} if it does not admit any transverse family of non-degenerate subtrees
apart from the trivial one (i.e.\ $\caly=\{T\}$).  A transverse family in $T$ is a \emph{transverse covering} if in addition, every segment in $T$ is covered by finitely many subtrees $Y_1,\dots, Y_n$ in $\mathcal{Y}$ with $Y_i\cap Y_{i+1}\neq \es$ (see \cite[Definition~4.6]{Gui04}). 
The \emph{skeleton} of a transverse covering of $T$ is the bipartite
simplicial $G$-tree $S$ having one vertex $v_Y$ for each subtree $Y$
in the family (the stabilizer of $v_Y$ in $S$ is the stabilizer of $Y$
in $T$), and one vertex $v_x$ for each point $x\in T$ belonging to at
least two subtrees of the family (the stabilizer of $v_x$ in $S$ is the stabilizer of $x$ in $T$); the vertex $v_x$ is joined to $v_Y$ by an edge if and only if $x\in Y$.

Every geometric tree with dense orbits has a unique transverse covering by a family
$\caly$ of non-degenerate indecomposable subtrees with finitely
generated stabilizers, see
e.g. \cite[Proposition~1.25]{Gui}. Moreover, each $Y\in \caly$ is
itself dual to a band complex $\Sigma_Y$ such that $\Sigma_Y/G_Y$ is a
finite band complex where every leaf is dense.
The \emph{skeleton} of $T$ is the skeleton of this transverse covering. 
The band complexes $\Sigma_Y$ are called the \emph{minimal components} of $\Sigma$.

\subsection{Boundary amenability}\label{sec-amenability}

We now review a construction, due to Ozawa, that allows for an inductive argument to prove that a discrete countable group $\Gamma$ is boundary amenable. We first recall the definition of topological amenability of a group action on a compact space. The space $\mathrm{Prob}(\Gamma)$ of probability measures on $\Gamma$ is equipped with the topology of pointwise convergence, or equivalently, the subspace topology from
$\ell^1(\Gamma)$.

\begin{de}[Topologically amenable action]\label{dfn_topo_amenable}
Let $\Gamma$ be a   countable group.  An action of $\Gamma$ by homeomorphisms on a compact Hausdorff space $X$ is \emph{topologically amenable} if there exists a sequence of continuous maps $$\mu_n: X \to\mathrm{Prob}(\Gamma)$$ such that for all $\gamma\in\Gamma$, one has $$\sup_{x\in X}||\mu_n(\gamma.x)-\gamma.\mu_n(x)||_1\to 0$$ as $n$ goes to $+\infty$.  
\end{de}

\begin{de}[Boundary amenability]
  A countable group $\Gamma$ is \emph{boundary amenable}   if it admits a topologically amenable action on a   nonempty compact Hausdorff space.
\end{de}

In particular, the trivial action of $\Gamma$ on a point is topologically amenable if and only if $\Gamma$ is amenable,   so amenable groups are always boundary amenable. There is a similar notion in the   measurable category, see \cite[Definition~2.11]{JKL}.   In the following definition, the space $\Prob(\Gamma)$ is equipped with its Borel $\sigma$-algebra, which (as $\Gamma$ is assumed countable) is also the $\sigma$-algebra generated by all evaluation maps $\mu\mapsto\mu(\gamma)$ with $\gamma\in\Gamma$, see e.g.\ \cite[Theorem~17.24]{Kec}.

\begin{de}[Borel amenable action]\label{Borel-def}
A $\Gamma$-action on a   measurable space $X$ is \emph{Borel amenable} if there exists a sequence of measurable maps $$\mu_n: X \to\mathrm{Prob}(\Gamma)$$ such that for all $\gamma\in \Gamma$ and all $x\in X$, one has $$||\mu_n(\gamma.x)-\gamma.\mu_n(x)||_1\to 0$$ as $n$ goes to $+\infty$.  
\end{de}

  It turns out that an action of a countable group $\Gamma$ by homeomorphisms on a compact Hausdorff topological space $X$ is topologically amenable if and only if it is Borel amenable, when $X$ is equipped with its Borel $\sigma$-algebra (see \cite{Ren} or Remark \ref{rk_eqv} below).

\paragraph*{The inductive procedure for proving that a group is   boundary amenable.}
In the sequel, all topological spaces will always be equipped with their Borel $\sigma$-algebra, when viewed as  measurable spaces.

\begin{prop}(Ozawa \cite[Proposition 11]{Oza})\label{ozawa} 
Let $\Gamma$ be a countable group, let $X,Y$ be two compact Hausdorff spaces equipped with  actions of $\Gamma$ by homeomorphisms, and let $K$ be a countable discrete space equipped with  an action of $\Gamma$. Assume that there exists a sequence of   measurable maps $$\mu_n:X\to\mathrm{Prob}(K)$$ such that for all $\gamma\in\Gamma$ and all $x\in X$, one has $$||\mu_n(\gamma.x)-\gamma.\mu_n(x)||_1\to 0$$ as $n$ goes to $+\infty$. Assume in addition that for all $k\in K$, the restriction of the $\Gamma$-action on $Y$ to the stabilizer $\mathrm{Stab}(k)\subseteq\Gamma$ is a topologically amenable $\mathrm{Stab}(k)$-action.
\\ Then the $\Gamma$-action on the compact space $X\times Y$ is topologically amenable, in particular $\Gamma$ is   boundary amenable.
\end{prop}

\begin{rk}\label{rk_eqv}
  The proposition implies that if $X$ is compact Hausdorff, and $\Gamma\actson X$ is Borel amenable then it is topologically amenable by taking
$K=\Gamma$ and $Y$ a point.
\end{rk}

\begin{proof}[Proof of Proposition~\ref{ozawa}]
The statement given in \cite[Proposition~11]{Oza} requires the convergence $$||\mu_n(\gamma.x)-\gamma.\mu_n(x)||_1\to 0$$ to be uniform in $x$, but the proof only requires that $$\int_{X}||\mu_n(\gamma.x)-\gamma.\mu_n(x)||_1 dm(x)\to 0$$ for all probability measures $m$ on $X$ and all $\gamma\in \Gamma$, see also \cite[Proposition~5.2.1]{BO}. By noticing that $||\mu_n(\gamma.x)-\gamma.\mu_n(x)||_1\le 2$ and using Lebesgue's dominated convergence theorem, this hypothesis can be replaced by pointwise convergence.
\end{proof}

The following consequence of Proposition~\ref{ozawa} is established in \cite[Proposition~C.1]{Kid}, by noticing that the restricted action of any   boundary amenable subgroup of $\Gamma$ on the Stone--\v{C}ech compactification $\beta\Gamma$ is topologically amenable.

\begin{cor}(Ozawa, Kida \cite[Proposition C.1]{Kid})\label{ozawa-kida}
Let $\Gamma$ be a countable group, let $X$ be a compact Hausdorff space equipped with  an action of $\Gamma$ by homeomorphisms, and let $K$ be a countable space equipped with  an action of  $\Gamma$. Assume that there exists a sequence of   measurable maps $$\mu_n:X\to\mathrm{Prob}(K)$$ such that for all $\gamma\in\Gamma$ and all $x\in X$, one has $$||\mu_n(\gamma.x)-\gamma.\mu_n(x)||_1\to 0$$ as $n$ goes to $+\infty$. Assume in addition that for all $k\in K$, the stabilizer $\mathrm{Stab}(k)\subseteq\Gamma$ is   boundary amenable.
\\ Then $\Gamma$ is   boundary amenable.
\end{cor}

Arguing as in \cite[Proposition~11]{Oza} without the hypothesis that $X$ is compact, one also gets 
that the analogous statement holds in   the measurable setting. In particular, we get the following.

\begin{prop}\label{amen}
  Let $\Gamma$ be a countable group, let $X$ be a   topological space equipped with an  action of $\Gamma$ by homeomorphisms,
  and let $K$ be a countable space equipped with  an action of $\Gamma$. Assume that there exists a sequence of   measurable maps $$\mu_n:X\to\mathrm{Prob}(K)$$ such that for all $\gamma\in\Gamma$ and all $x\in X$, one has $$||\mu_n(\gamma.x)-\gamma.\mu_n(x)||_1\to 0$$ as $n$ goes to $+\infty$. Assume in addition that for all $k\in K$, the stabilizer $\mathrm{Stab}(k)\subseteq\Gamma$ is amenable.
\\ Then the $\Gamma$-action on $X$ is Borel amenable.
\qed
\end{prop}

\paragraph*{Stability under subgroups.} 
The following well-known fact says that boundary amenability is stable under passing to subgroups. It can be viewed as a consequence of Ozawa's result (Proposition~\ref{ozawa}), applied with $K=\Gamma$ and $Y$ a point.

\begin{cor}\label{cor-subgroup}
Let $\Gamma$ be a countable group, let $\Delta\subseteq\Gamma$ be a subgroup, and let $X$ be a compact Hausdorff space equipped with a topologically amenable $\Gamma$-action.
\\ Then the $\Delta$-action on $X$ 
obtained by restriction of the $\Gamma$-action is topologically amenable, so in particular $\Delta$ is   boundary amenable.
\\ In particular, every point stabilizer for the $\Gamma$-action on $X$ is amenable.
\qed
\end{cor}

\begin{rk}\label{rk-subgroup}
The same statement applies in the   measurable setting. In particular, 
if the $\Gamma$-action on $X$ is Borel amenable, then every point stabilizer in $\Gamma$ is amenable.
\end{rk}

\paragraph*{Stability under extensions.}  This was proved by Kirchberg--Wassermann in \cite{KW}.

\begin{prop}[Kirchberg--Wassermann \cite{KW}]\label{ext-exact}
Any extension of two countable groups   which are boundary amenable is   boundary amenable. 
\qed
\end{prop}

If one wants to keep track of spaces, one can use Ozawa's result to prove the following.

\begin{cor}\label{cor-xy}
Let $$1\to H\to\Gamma\to Q\to 1$$ be a short exact sequence of countable groups. Let $X$ be a compact Hausdorff space equipped with an  action of $Q$ by homeomorphisms, and $Y$ be a compact Hausdorff space equipped with  an action of $\Gamma$ by homeomorphisms. Assume that the action $Q\actson X$ and the restricted action $H\actson Y$ are topologically amenable.
\\ Then the $\Gamma$-action on $X\times Y$ is topologically amenable (where the $\Gamma$-action on $X$ is the one factoring through $Q$).  
\end{cor}

\begin{proof}
Since the $Q$-action on $X$ is topologically amenable, there exists a sequence of   measurable maps $$\mu_n:X\to\mathrm{Prob}(Q)$$ such that for all $x\in X$ and all $\gamma\in\Gamma$, we have $$||\mu_n(\gamma.x)-\gamma.\mu_n(x)||_1\to 0$$ as $n$ goes to $+\infty$. The corollary then follows from Proposition~\ref{ozawa} applied to $K=Q$.
\end{proof}

\paragraph*{Stability under finite-index overgroups.} The following well-known easy fact implies that boundary amenability is a commensurability invariant; its proof is left to the reader.

\begin{prop}\label{fidx-exact}
Let $\Gamma$ be a group, and let $X$ be a compact Hausdorff space equipped with an action of $\Gamma$ by homeomorphisms. Let $\Gamma^0$ be a finite index subgroup of $\Gamma$, such that the $\Gamma^0$-action on $X$ is topologically amenable.
\\ Then the $\Gamma$-action on $X$ is topologically amenable.
\end{prop}

\paragraph*{Quotient by an amenable subgroup.}

\begin{prop}[Nowak \cite{Now}]\label{prop_quotient_amenable}
Let $G$ be a countable group, and $N\triangleleft G$ an amenable normal subgroup. Then $G$ is   boundary amenable if and only if $G/N$ is   boundary amenable.  
\\ In particular, $G$ is   boundary amenable if and only if $G/Z(G)$ is   boundary amenable.
\end{prop}

\begin{proof}
  One implication follows from stability under extension. The other implication was proved in \cite{Now}.
This is stated for $G$ finitely generated but extends to the general case since   boundary amenability is closed under increasing unions.
\end{proof}

\section{The factorization lemma, the case of free actions, and reduction to naming}\label{sec-scheme}

In the present section, we prove amenability of the action of $\Out(F_N)$ on the set of free arational $F_N$-trees (Theorem \ref{thm:FN-case}).
This result is in fact a particular case of Proposition~\ref{name-enough}
which is the key result
 allowing us to reduce the proof of boundary amenability of $\Out(F_N)$
to a technical statement about arational trees that will be established in the next section.

 We first prove a \emph{factorization lemma} (Lemma~\ref{lem_blow_up} below) saying that given two morphisms $f$ and $f'$ having their sources in Outer space and targeting the same arational tree $T$,
if $f$ and $f'$ take the same set of turns in $T$ (in which case we say that $f$ and $f'$ have the same \emph{turning class}),
then $f$ and $f'$ factor through the same trees in Outer space close enough to $T$.
We then define probability measures $\mu_n(f)$ on the set of simplices of the Outer space by considering the trees through which $f$ factors, and the factorization lemma together with some finiteness statements will ensure that $\mu_n(f)$ asymptotically only depends on the turning class of $f$.

 For free arational $F_N$-trees, there are only finitely many orbits of turns, so finitely many possible turning classes.
We can therefore define probability measures $\mu_n(T)$ by averaging the measures $\mu_n(f)$ for a finite set
of morphisms representing these turning classes, and deduce amenability of the action of $\Out(F_N)$ on the set of free arational $F_N$-trees.

In a more general situation, we have to assume that we have a finite set of preferred turning classes on every arational tree $T$
to run the same argument. 
This is the contents of Proposition~\ref{name-enough}: turning classes are given a name,
and the first name of all turning classes of all morphisms with target $T$ defines our preferred set.
These names 
will be constructed in Section \ref{sec-name} via a case-by-case analysis of geometric and non-geometric trees arational trees. 

\subsection{Turning classes and the factorization lemma}

The following definition, which is central in the present work, was inspired by a preprint of Los--Lustig \cite{LL} (where it is called a \emph{blow-up class}).
Recall that a turn at a point $x$ in a tree $T$ is a pair $(d,d')$ where $d,d'$ are two distinct directions at $x$.

\begin{de}[Turning class]
Let $T\in\baro$. A \emph{turning class} in $T$ is a $G$-invariant collection of turns at the branch points of $T$. 
\end{de}

An important example of turning classes is the following. Given a morphism $f:S\to T$, recall that a subtree $Y\subseteq S$ is \emph{legal} 
if $f$ is an isometry when restricted to $Y$. 
A turn $(d,d')$ in $S$ based at $x\in S$ is \emph{legal} if there exist small intervals $[x,y],[x,y']$
representing respectively the directions $d,d'$ such that the interval $[y,x]\cup [x,y']$ is legal.

\begin{de}[Turning class of a morphism]
Let $S\in\Os$, let $T\in\baro$, and let $f:S\to T$ be an optimal morphism.
\\ The \emph{turning class of $f$} is the collection of all orbits of turns $(d,d')$ at branch points of $T$ such that there exists a legal segment $I\subseteq S$ whose $f$-image crosses the turn $(d,d')$. 
\end{de}

The first important ingredient in our proof of boundary amenability of $\Out(G,\calf^{(\mathrm{t})})$ is the following factorization lemma. 

\begin{lemma}[Factorization lemma]\label{lem_blow_up}
Let $S,S'\in\Os$, let $T\in\baro$ with trivial arc stabilizers, and let $f:S\to T$ and $f':S'\to T$ be optimal morphisms. Assume that $f$ and $f'$ have the same turning class.
\\ Then there exists $\epsilon>0$ such that whenever $U\in \Os$ is a tree such that $f'$ factors through $U$, and such that the induced morphism $f'_U:U\to T$ has BBT smaller than $\epsilon$, then $f$ also factors through $U$.
\end{lemma}

\begin{rk}
It is actually enough to assume that the turning class of $f$ is contained in the turning class of $f'$.
\end{rk}

\begin{rk}\label{rk-vol}
The hypothesis $\BBT(f'_U)<\epsilon$ can be replaced by the assumption that the volume of the quotient graph $U/G$ is at most $\epsilon$: indeed, the Lipschitz constant of the morphism $f'_U$ is equal to $1$, so we have $\BBT(f'_U)\le \vol(U/G)$ by \cite[Lemma~3.1]{BFH} or \cite[Proposition~3.12]{Hor3}.
\end{rk}

\begin{proof}
We say that a nondegenerate legal segment $I\subseteq S$ is \emph{$S'$-liftable} if its image in $T$ lifts isometrically to $S'$, i.e.\ there exists an $f'$-legal segment $I'\subseteq S'$ such that $f(I)=f'(I')$.
\\
\\
\textbf{Claim 1:} There exists $\epsilon_0>0$ such that 
\begin{itemize}
\item every point in $S$ is the midpoint of an $S'$-liftable legal segment of length $2\epsilon_0$,
\item for every vertex $v$ of finite valence in $S$, every legal segment of length $2\epsilon_0$ centered at $v$ is $S'$-liftable,
\item for every vertex $v$ of infinite valence in $S$ and every legal interval $I$ of length $\epsilon_0$ with endpoint $v$, there exists a nontrivial element $g\in G_v$ such that $I\cup gI$ is $S'$-liftable (here $G_v$ is the stabilizer of the vertex $v$).
\end{itemize}  

\noindent \textbf{Proof:} We first observe that for every $x\in S$,
every nondegenerate legal segment $I$ in $S$ centered at $x$ contains
a nondegenerate $S'$-liftable subsegment centered at $x$. Indeed, the
turn in $T$ defined by the image of the two directions at $x$ in $I$
is either based at a point of valence $2$ in $T$, or else it is
contained in the turning class of $f$, hence of $f'$. So there exists
a neighborhood $V_x\subseteq I$ of $x$ in $I$ whose image in $T$ lifts
isometrically to $S'$.

The second assertion of the claim follows immediately since there are only finitely many turns involved. To prove the third, let $e$ be an edge incident on a vertex $v$ of infinite valence in $S$. Let $g\neq 1$ be an arbitrary element of $G_v$. The turn $(e,ge)$ is legal because $T$ has trivial arc stabilizers. Applying the above observation shows that there exists $\epsilon_0>0$ such that the $\epsilon_0$-neighborhood of $v$ in the segment $e\cup ge$ is $S'$-liftable. The third assertion follows because there are only finitely many orbits of edges. 

We now prove the first assertion. Let $l$ be the minimal length of an edge in $S$. Since $f$ is optimal, we can find legal segments $J_1,\dots,J_k$ obtained by concatenation of two edges of $S$ such that the orbit of every point in $S$ intersects some $J_i$ in a point at distance at least $l/10$ from $\partial J_i$. We are going to use a compactness argument in the disjoint union $J_1\dunion\dots\dunion J_k$. For every $i\in\{1,\dots,k\}$ and every $x\in J_i$, the observation made in the first paragraph of the proof provides a neighborhood $V_x$ of $x$ in $J_i$ that is $S'$-liftable. Let $\epsilon_1$ be a Lebesgue number for this open covering of $J_1\dunion\dots\dunion J_k$. Without loss of generality, we can assume that $\epsilon_1<l/10$. Since every orbit meets one of the segments $J_i$ at distance at most $l/10$ from $\partial J_i$, this proves the first assertion with $\epsilon_0=\epsilon_1/2$.   
\\
\\  
\indent Let now $U\in \Os$ be a tree such that there exist morphisms $f'_{S'U}:S'\to U$ and $f'_U:U\to T$, with $\BBT(f'_U)\le\epsilon_0/100$, such that the following diagram commutes:
$$\xymatrix@C=1cm@R=1cm{
S'\ar[r]^{f'_{S'U}}\ar@/_1pc/[rr]_{f'}&U\ar[r]^{f'_U}&T.
}
$$
We are going to define a map $f_{SU}$ from $S$ to $U$ in the following way. We first note that any $S'$-liftable segment is $U$-liftable. Now given $x\in S$, choose $I$ a $U$-liftable legal segment of length $2\epsilon_0$ centered at $x$, and choose a lift $\tilde I$ of $f(I)$ in $U$ (which naturally comes with an isometry $j:I\to\tilde I$); we want to map $x$ to the midpoint of $\tilde I$. We will prove that this definition is independent of choices, and that this defines an  optimal morphism from $S$ to $U$. It will then be obvious that $f$ factors through this map.
\\
\\
\textbf{Claim 2:} If $I\subseteq S$ is an $S'$-liftable segment of
length greater than $2\epsilon_0/100$, and if $\tilde I_1$ and $\tilde
I_2$ are two lifts of $f(I)$ in $U$, then the isometries
$j_1:I\to\tilde I_1$ and $j_2:I\to \tilde I_2$ coincide on the
complement of the $\epsilon_0/100$-neighborhood of the endpoints of
$I$.
\\ \textbf{Proof:} Note that for all $x\in I$, the points $j_1(x)$ and $j_2(x)$ have the same $f'_U$-image in $T$. Write $I=[a,b]$. Let $x\in I$ be a point at distance greater than $\epsilon_0/100$ from $\partial I$. For all $i\in\{1,2\}$, we denote by $a_i,b_i,x_i$ the $j_i$-images of $a,b,x$. Let us prove that $x_1=x_2$.  
If $x_1\neq x_2$, then $x_1\notin \tilde I_2$ and $x_2\notin \tilde I_1$ because
$x_1$ and $x_2$ have the same image in $T$, and both segments $\tilde I_i$ embed in $T$.
It follows that $x_1$ lies in $[a_1,a_2]$ or $[b_1,b_2]$ (say $[a_1,a_2]$): indeed, if $[a_1,a_2]\cup [b_1,b_2]$ differs from $K=\mathrm{Hull}(a_1,a_2,b_1,b_2)$, then the complement
of $[a_1,a_2]\cup [b_1,b_2]$ in $K$ is contained in $[a_1,b_1]\cap [a_2,b_2]$, which cannot contain $x_1$. 
 However $f'_U(a_1)=f'_U(a_2)$, while $d_T(f'_U(x_1),f'_U(a_1))=d_U(x_1,a_1)>\BBT(f'_U)$ by assumption, contradicting the definition of the BBT.
\\
\\
\textbf{Claim 3:} Let $v$ be a vertex of infinite valence in $S$, and let $\tilde v$ be the point fixed by $G_v$ in $U$. Let $I\subseteq S$ be a legal segment that contains $v$, and such that $v$ is at distance at  least $2\epsilon_0/100$ from $\partial I$. Then any lift of $I$ to $U$ passes through $\tilde v$.
\\\textbf{Proof:} Write $I=I_1\cup I_2$, where $I_1$ and $I_2$ are two segments of length at least $2\epsilon_0/100$ with endpoint $v$. By Claim~1, there exists an element $g_1\neq 1$ in $G_v$ such that $I_1\cup g_1I_1$ is $S'$-liftable (hence $U$-liftable); let $I^*_1:=g_1I_1$, and $\Tilde J_1$ 
be a lift of $I_1\cup I^*_1$ in $U$ with  an isometry $j_1:I_1\cup I^*_1\ra \Tilde J_1$, and let $\Tilde I_1=j_1(I_1)$, $\Tilde I^*_1=j_1(I^*_1)$. 
Denote by $m$ the midpoint of $\Tilde I_1$.  
Notice that $g_1\Tilde I_1$ is also a lift of $I_1^*=g_1I_1$ in $U$, so by Claim~2 it has the same midpoint as $\Tilde I^*_1$. In other words $g_1$ sends the midpoint of $\Tilde I_1$ to the midpoint of $\Tilde I^*_1$. Therefore, the common endpoint of $\tilde I_1$ and $\Tilde I^*_1$ (which is also the midpoint of $[m,g_1m]$) is equal to $\tilde{v}$. This shows that $I_1$ has a lift in $U$ with endpoint $\tilde v$, and by symmetry $I_2$ also has a lift in $U$ with endpoint $\tilde v$. Their concatenation gives a lift $\tilde I$ of $I$ in $U$ that contains $\tilde v$, and such that $\tilde v$ is at distance at least  $2\epsilon_0/100$ from $\partial\tilde I$. Claim~2 then implies that all lifts of $I$ in $U$ pass through $\tilde v$,  so Claim~3 is proved.
\\
\\
\indent Let now $x\in S$, and let $I_1$ and $I_2$ be two segments of length  $\epsilon_0$ centered at $x$. We are left showing that if $j_1:I_1\to\tilde I_1$ and $j_2:I_2\to\tilde I_2$ are lifts of $I_1$ and $I_2$ in $U$, then $j_1(x)=j_2(x)$. The fact that $f_{SU}$ is locally isometric when restricted to legal segments, hence an  optimal morphism, will then follow from Claim~2.  From now on, we will assume that $\epsilon_0$ has been chosen smaller than the lengths of the edges of $S$
and small enough so that if two directions $d_1,d_2$ form an illegal turn at a vertex $v\in S$, then there are subsegments $[v,x_1]\subseteq d_1$ and $[v,x_2]\subseteq d_2$ of length at least $\epsilon_0$ that have the same image in $T$: this is possible because there are only finitely many orbits of illegal turns in $S$, otherwise $T$ would have nontrivial arc stabilizers. 

If $x$ is at distance at least $\epsilon_0/100$ from all branch points of $S$, then $I_1$ and $I_2$ contain a common subsegment of length $2\epsilon_0/100$ centered at $x$, and it follows from Claim~2 that all lifts $j_1:I_1\to\tilde I_1$ and $j_2:I_2\to\tilde I_2$ in $U$ have the same midpoint. 

We can thus assume that $x$ is at distance at most $\epsilon_0/100$ from some branch point $v$ of $S$. We write $I_1=J_1\cup J^*_1$ with $J_1\cap J^*_1=\{v\}$, and $x\in J_1$. Similarly, we write $I_2=J_2\cup J^*_2$ with $J_2\cap J^*_2=\{v\}$, and $x\in J_2$ (in particular $J_1=J_2$). If the tripod formed by $J_1,J^*_1$ and $J^*_2$ is illegal, then  the hypothesis we made on $\epsilon_0$ ensures that $I_1$ and $I_2$ have the same image in $T$, and hence they have common lifts in $U$. The result then follows as above from Claim~2. We now assume that this tripod is legal. 

For all $i\in\{1,2\}$, denote by $m_i,m^*_i$ the midpoints of $J_i,J^*_i$ respectively, and let $\tilde m_i,\Tilde m^*_i$ be their images in $\tilde I_i$. 
Since all three segments $J_1$, $J_1^*$ and $J^*_2$ have length at least $2\epsilon_0/100$, Claim~2 implies that for all $i\in\{1,2\}$, the lifts $\tilde m_i$ and $\tilde m_i^*$ do not depend on the choice of a lift $\tilde I_i$ of $I_i$, and in addition $\tilde m_1=\tilde m_2$. Denote by $\tilde c$ the center of the tripod $[\tilde m_1,\tilde m_1^*,\tilde m_2^*]$ in $U$. To complete the proof, it is enough to check that $d_{U}(\tilde m_1,\tilde c)\ge d_{S}(m_1,v)$,  because this implies that the images of $x$ in $\tilde I_1$ and in $\tilde I_2$ are the same. If $v$ has finite valence, this follows from the second assertion of Claim~1 which implies that $J^*_1\cup J^*_2$  
is also $U$-liftable and contains $[m^*_1,m^*_2]$, so that $(\Tilde m_1,\Tilde m^*_1,\Tilde m^*_2)$ is an isometric lift of $(m_1,m^*_1,m^*_2)$. 
If $v$ has infinite valence, then this is a consequence of Claim~3, showing that $\Tilde I_1$ and $\Tilde I_2$ both contain the point $\Tilde v$ of $U$ fixed by $G_v$ so $\Tilde v\in [\tilde m_1,\tilde c]$.    
\end{proof}

\subsection{Probability measures associated to a morphism}\label{sec-finite-width}

Given a morphism $f:S\to T$ from a simplicial metric tree $S\in\calo$ to an $\mathbb{R}$-tree $T\in\AT$, and a real number $t$, we let $\mathcal{S}_t(f)$ be the collection of all simplices $\Delta\in\Simp$ such that there exists a tree $U\in\calo$ of covolume $e^{-t}$, whose image in $\mathbb{P}\Os$ is contained in $\Delta$, such that $f$ factors through $U$. 
The goal of the present section is to associate to the morphism $f$ a sequence of probability measures $\mu_n(f)$ on $\Simp$, obtained by averaging uniform measures on the sets $\cals_t(f)$. The key proposition we need to establish is the following.

\begin{prop}\label{sphere-finite}
Let $T\in\AT$, let $S\in\Os$, and let $f:S\to T$ be a morphism.
\\ Then for all $t\in\mathbb{R}$, the set $\cals_t(f)$ is finite.
\end{prop} 

Before proving the proposition, we will prove two preliminary lemmas.
Given a tree $S\in\calo$, we define the \emph{systole} of $S$ as the smallest translation length in $S$ of a nonperipheral element of $G$.

\begin{lemma}\label{lem_systole}
  Let $\lambda>0$. Let $T\in\baro$ be a tree with trivial arc stabilizers.
 Let $f:S\ra T$ be a morphism, and let $(U_i)_{i\in\mathbb{N}}\in \calo^{\mathbb{N}}$ be a sequence of trees of systole at least $\lambda$ such that
$f$ factors through all trees $U_i$.
\\ Then the set of simplices in $\Simp$ to which the trees $U_i$ project is finite.
\end{lemma}

\begin{proof}
Assume by contradiction that the set of simplices spanned by $\{U_i\}_{i\in \bbN}$ is infinite.
Let $f_i:S\ra U_i$ be such that $f$ factors through $f_i$.
Consider the $f_i$-preimage $V_i\subset S$ of the set of vertices of $U_i$.

We claim that for each edge $e$ of $S$   (for the natural simplicial structure on $S$ coming where edges correspond to complementary components of generalized branch points), $\# (e\cap V_i)$ is bounded.
Otherwise, there exists oriented subedges $I,J\subset e$ (bounded by points in $V_i$)
that are at distance at most $\lambda/2$ from one another and whose images in $U_i$ are two oriented edges in the same orbit under some element $g_i$.
Since $f_i$ is an isometry when restricted to $e$, the translation length of $g_i$ in $U_i$ is positive and at most $\lambda/2$, contradicting our the hypothesis on the systole. This proves the claim.

Up to passing to a subsequence, we can therefore assume that the combinatorics of the subdivision defined by $V_i$ does not depend on $i$.
We denote by $S^0$ this combinatorial tree.

For each $i\in \bbN$, choose a pair of adjacent edgelets folded by $f_i$, and let $S^1_i$ be the combinatorial tree obtained by folding these two edgelets together. We claim that only finitely many combinatorial trees $S^1_i$ appear.
Otherwise, there exist  two edgelets $e,e'\subset S$ sharing a vertex $v$, and infinitely many elements $g_i\in G_v$ such that $e$ is identified with $g_i e'$.
Note that the precise (metric) subsegment of $S$ corresponding to $e$ depends on $i$. We denote by $E,E'$ the edges of $S$ (for its natural set of vertices) containing
$e$ and $e'$ respectively ($E$ and $E'$ do not depend on $i$).
Then $f(E)$ and $f(g_i E')$ share a common subsegment, so for all $i,j\in\mathbb{N}$, the element $g_ig_j\m$ fixes an arc in $T$, a contradiction.

Thus, up to extracting a subsequence, we can assume that $S^1_i=S^1$ does not depend on $i$ (as a combinatorial tree).
We now assume that we have constructed by induction combinatorial trees $S^1$, $S^2$, \dots, $S^k$
such that  for all $j< k$, the tree $S^{j+1}$ is obtained from $S^j$ by folding two edgelets,
and such that for each $i$,  
$f_i$ factors through 
these folds. 
A coherent set of metrics on $S^0,S^1,\dots,S^k$ is a metric for each $S^j$ such that
$S^0$ becomes isometric to $S$, and each $S^j\ra S^{j+1}$ is a morphism.
By induction we will also have that for each $i$ there exists a coherent set of metrics such that 
$f_i:S\ra U_i$ factors through $S^1,\dots,S^k$, hence so does $f:S\ra T$. 

To construct $S^{k+1}$, for each $i\in \bbN$, choose a pair of adjacent edgelets of $S^k$ folded by $f_i$ (they exist for infinitely many $i$ because we assume that
the set of simplices spanned by $\{U_i\}_{i\in \bbN}$ is infinite).
For each $i$, let $S^{k+1}_i$ be the combinatorial tree obtained by folding these two edgelets together.
As above, we claim that only finitely many combinatorial trees $S^{k+1}_i$ appear.
Indeed, if not, there exist two edgelets $e,e'\subset S^k$ sharing a vertex $v$, and infinitely many elements $g_i\in G_v$ such that $e$ is identified with $g_i e'$.
Let $\Tilde E, \Tilde E'$ be natural edges of $S$ containing a preimage of $e$ and $e'$ respectively.
Then $f(\tilde E)$ and $f(g_i \tilde E')$ share a common subsegment. Since $g_i\in G_v$ and arc stabilizers of $T$ are trivial, there are at most 4 indices $i$ such that $f(\tilde E)\cap g_if( \tilde E')$ is non-degenerate, a contradiction.
This proves the claim and allows to construct $S^{k+1}$. 

This allows to construct an infinite sequence of trees $S^1,\dots,S^k,\dots$.
Since $S^j$ has fewer orbits of edgelets than $S^{j-1}$, this is a contradiction. This completes the proof of the lemma.
\end{proof}

\begin{lemma}\label{lem_systole2} There exist $K,K'>0$ such that for every $\lambda>0$, the following holds.
\\ Let $U\in\calo$ be a tree of systole at least $\lambda>0$, and let $f:U\ra U'$ be a morphism with
$\vol(U'/G)\geq \vol(U/G)-\frac{\lambda}{K}$.
\\ Then the systole of $U'$ is at least $\lambda/K'$.
\end{lemma}

\begin{proof}
  Let $g\in G$ be an element which is hyperbolic in $U'$. Its axis $l$ in $U$ contains 
an edge of length at least $\lambda/4N$, where $N$ is the number of orbits of edges of $U$.
Otherwise, $l$ would contain two oriented edges in the same orbit, and separated by at most $2N$ edges, so $S$ would have
a hyperbolic element of translation length at most $\lambda/2$, a contradiction.

Let $e$ be an edge in $l$ of length at least $\lambda/4N$, and let $K=16N$, and $K'=8N$.
Let $[u,v]\subset e$ be the central subsegment of length $\lambda/8N$, oriented so that $u,v,gu$ are in this order.
Since $\vol(U'/G)\geq \vol(U/G)-\frac{\lambda}{16N}$, no point in $[u,v]$ is identified with any other point in $U$.
In particular, $f([v,gu])\cap f([u,v])=\{f(v)\}$ and $f([g\m v,u])\cap f([u,v])=\{f(u)\}$.
This implies that $g\m f(v), f(u) ,f(v),gf(u)$ are aligned in this order in $U'$. Therefore $[f(u),f(v)]$ is contained in a fundamental domain
of the axis of $g$ in $U'$, so $||g||_{U'}\geq \lambda/8N$.
\end{proof}

\begin{proof}[Proof of Proposition \ref{sphere-finite}]
Assume towards a contradiction that $\cals_t(f)$ is infinite for some $t\in\mathbb{R}$.
Given a set of trees $\calu\subset \calo$, we denote by $\Simp(\calu)\subset \Simp$ the set of simplices 
obtained by forgetting the metric on the trees in $\calu$.

Let $\calu_0$ be the set of trees of covolume $e^{-t}$ through which $f$ factors.  
By assumption, the set $\Simp(\calu_0)$ is infinite. We denote by $\Delta_0$ the simplex containing $S$. For each $U\in\calu_0$, we let $S_{U,0}:=S$, and we choose a morphism $f_{U,0}:S\to U$ through which $f$ factors.

For every $i\in\mathbb{N}$, we are going to define inductively (see Figure~\ref{fig-newton}) a simplex $\Delta_i$, real numbers $v_i>e^{-t}$, $\lambda_i>0$, 
a set of trees $\calu_i\subseteq \calu_{i-1}$, and for each $U\in \calu_i$,  
a tree $S_{U,i}$ projecting to $\Delta_i$ and a morphism $f_{U,i}:S_{U,i}\to U$ such that the following holds:
\begin{itemize}
\item $\Simp(\calu_i)$ is infinite,
\item $\vol(S_{U,i}/G)\leq v_i$ and $\mathrm{systole}(S_{U,i})\in[ \lambda_i,2\lambda_i]$ for all $U\in \calu_i$,
\item $f_{U,i-1}:S_{U,i-1}\ra U$ factors through $f_{U,i}:S_{U,i}\ra U$.
\end{itemize}
We take for $v_0$ and $\lambda_0$ the covolume and the systole of $S$.

\begin{figure}[htb]
\begin{center}
\includegraphics[width=.8\textwidth]{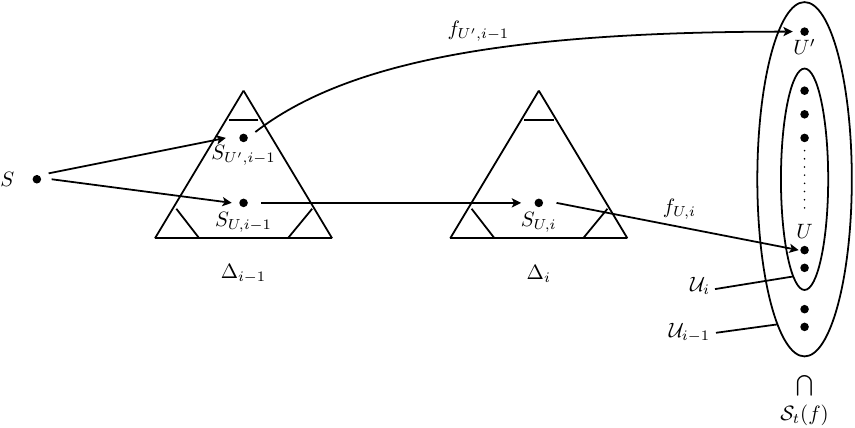}
\caption{The construction from the proof of Proposition~\ref{sphere-finite}. The corners in the simplices are drawn to represent the subspace made of trees with a systole bounded from below.}
\label{fig-newton}
\end{center}
\end{figure}

Let $v_{i+1}=v_i-\frac{\lambda_i}{K}$, where $K$ is the integer from Lemma~\ref{lem_systole2}.
If $e^{-t}\geq v_{i+1}$, then Lemma~\ref{lem_systole2} implies that the systole of the trees in $\calu_i$ is bounded from below.
Thus Lemma \ref{lem_systole} implies that $\Simp(\calu_i)$ is finite, a contradiction.

Thus $e^{-t}< v_{i+1}$, so
for each $U\in \calu_i$, one can choose a tree $S_{U,i+1}$ of covolume $v_{i+1}$ through which $f_{U,i}$ factors,
and let $f_{U,i+1}: S_{U,i+1}\to U$ be the induced morphism. 
Lemma~\ref{lem_systole2} implies that  the systole of the trees $S_{U,i+1}$ is bounded from below as $U$ varies in $\calu_i$.
  Lemma \ref{lem_systole} shows that as $U$ varies in $\calu_i$, the trees $S_{U,i+1}$ project to finitely many simplices of $\mathbb{P}\calo$.
 Let $\Delta_{i+1}$ 
be a simplex such that $S_{U,i+1}\in \Tilde \Delta_{i+1}$ for all trees $U$ in a set $\calu'_{i+1}\subset \calu_i$ with $\Simp(\calu'_{i+1})$ infinite. 
By the pigeonhole principle, there exists $\lambda_{i+1}>0$ such that the set $\calu_{i+1}$ of trees $U\in\calu'_{i+1}$
such that $\mathrm{systole}(S_{U,i+1})\in[\lambda_{i+1},2\lambda_{i+1}]$ is such that $\Simp(\calu_{i+1})$ is infinite.
This concludes our inductive construction.

Since $v_{i+1}=v_i-\lambda_i/K$ and $v_i\geq e^{-t}$ for all $i$, we have $\lambda_i\ra 0$.
Consider for each $i\in \bbN$, $U_i\in\calu_i$. The tree $S_{U_i,j}$ is defined for $i\geq j$.  Its length function is bounded from below by the length function of $T$, and bounded from above   when $j$ is fixed because $\vol(S_{U_i,j}/G)\le v_j$. Therefore, up to diagonal extraction of a subsequence,
we can assume that for each $j$, $S_{U_i,j}$ converges to a tree $S_{\infty,j}\in \baro$.
By continuity of the systole and the covolume on the closure of $\Tilde \Delta_j$, we have $\mathrm{systole}(S_{\infty,j})\leq 2\lambda_j$, and $\vol(S_{\infty,j}/G)\geq e^{-t}$.
There are $1$-Lipschitz maps $S_{\infty,j}\ra S_{\infty,j+1}$ and $S_{\infty,j}\ra T$. 
The trees $S_{\infty,j}$ converge to some tree $S_{\infty,\infty}$ as $j\ra \infty$ and there is a $1$-Lipschitz map $S_{\infty,\infty}\ra T$.
By semi-continuity of the volume, we have $\vol(S_{\infty,\infty}/G)\geq e^{-t}$, 
so Lemma~\ref{morphism-arat-1} implies that $S_{\infty,\infty}$ is a Grushko tree. In particular, the systole of  $S_{\infty,\infty}$ is positive and 
bounds $2\lambda_j$ from below, a contradiction.
\end{proof}

\begin{rk}\label{rk-covol}
Notice that the same proof would also have worked if in the definition of $\cals_t(f)$, we had replaced the condition that the covolume of $U$ is equal to $e^{-t}$  by the condition that this covolume is at least $e^{-t}$.  

Notice also that for all $t$ 
sufficiently large (so that the source of $f$ has covolume at least $e^{-t}$), the set $\cals_t(f)$ is nonempty.
\end{rk}

Fix an arational tree $T\in \AT$ and $f\in \Opt_{\ra T}$ an optimal morphism with range $T$.
For all $t>0$ such that $\cals_t(f)$ is nonempty, we let $\nu_t(f)$ be the uniform probability measure on the finite set $\cals_t(f)$. For all $n\in\mathbb{N}$, we then define a probability measure $\mu_n(f)$   on $\Simp$ by letting $$\mu_n(f):=\frac{1}{n}\int_{n}^{2n}\nu_t(f)d\mathrm{Leb}(t)$$ if $\mathcal{S}_{e^{-n}}(f)\neq\emptyset$, and otherwise we just let $\mu_n(f)$ be some fixed probability measure $\mu_0$ on $\Simp$. As a consequence of the factorization lemma (Lemma~\ref{lem_blow_up}, see also Remark~\ref{rk-vol}),  if $f,g\in\Opt_{\ra T}$ have the same turning class, then $\cals_t(f)=\cals_t(g)$ for all sufficiently large $t$, and therefore we obtain the following fact.

\begin{cor}\label{cor-deco}
Let $T\in \AT$ be an arational tree. Let $f,g\in \Opt_{\ra T}$ be two optimal morphisms with range $T$ and having the same turning class.  
\\ Then $\mu_n(f)=\mu_n(g)$ for all sufficiently large $n\in\mathbb{N}$.
\qed
\end{cor}

In the sequel, in order to have an action of  $\Out(G,\calf)$ on a compact space, we will need to work with projective classes of trees instead of isometry classes of trees. The next lemma will ensure that the measures $\mu_n$ do not depend too strongly on a choice of representative in a projective class.  Given a tree $T\in\baro$ and $\lambda>0$, we denote by $\lambda.T$ the tree obtained from $T$ by dilating the metric by $\lambda$. Given an optimal morphism $f:S\to T$, we let $\lambda f:\lambda.S\to\lambda.T$ be the corresponding morphism. 

\begin{lemma}\label{shift}
Let $T\in \AT$, $f\in\Opt_{\ra T}$, and let $\lambda>0$. 
\\ Then $||\mu_n(f)-\mu_n(\lambda.f)||_1\to 0$ as $n$ goes to $+\infty$.
\end{lemma} 

\begin{proof}
We have 
\begin{eqnarray*}
\mu_n(\lambda.f)&=&\frac{1}{n}\int_{n}^{2n}\nu_t(\lambda.f)d\text{Leb}(t)\\
& =&\frac{1}{n}\int_{n}^{2n} \nu_{t+\log\lambda}(f)d\text{Leb}(t). 
\end{eqnarray*}
\noindent Therefore 
$$\mu_n(\lambda.f)-\mu_n(f) =  
\frac{1}{n}\int_{2n}^{2n+\log\lambda} \nu_{t'}(f)d\text{Leb}(t') -\frac{1}{n}\int_n^{n+\log\lambda}\nu_{t'}(f)d\text{Leb}(t'),$$
\noindent so  $$||\mu_n(f)-\mu_n(\lambda.f)||_1\le  \frac{2|\log\lambda|}{n}$$ and the conclusion follows.
\end{proof}

As a consequence of Corollary~\ref{cor-deco} and Lemma~\ref{shift}, we obtain the following fact.

\begin{cor}\label{cor-shift}
Let $T\in \AT$ and consider $f,g\in\Opt_{\ra T}$  two morphisms having  
the same turning class, and let $\lambda>0$. 
\\ Then $||\mu_n(f)-\mu_n(\lambda.g)||_1\to 0$ as $n$ goes to $+\infty$.
\qed 
\end{cor} 

\begin{lemma}\label{sphere-borel}
For every $n\in\mathbb{N}$, the map $\mu_n:\Opt_{\ra\AT}\to \mathrm{Prob}(\Simp)$ is Borel.
\end{lemma}

\begin{proof}
 Given $\tau\in\Simp$, let $Z_\tau\subset \bbR\times \Opt_{\ra\AT}$ be defined by
$$Z_\tau=\{(t,f)\in \bbR\times \Opt_{\ra\AT}\,|\ \tau\in \cals_t(f)\}.$$ 
It suffices to prove that $Z_\tau$ is Borel
since this  implies that the map  
$$\nu:\bbR\times(\Opt_{\ra\AT})\to \mathrm{Prob}(\Simp)$$
sending $(t,f)$ to $\nu_t(f)$ is Borel, and the lemma follows.

Given a simplex $\sigma\in \Simp$ 
and $\eps,v >0$, denote by $\Tilde \sigma_{v,\eps}$ the subset of unprojectivized Outer space $\calo$ defined as the closure of the set of trees projecting to $\sigma$, with covolume at most $v$,
and systole at least $\eps$. This is a compact set and the space $\Opt_{\Tilde \sigma_{v,\eps}}$  of all optimal morphisms $f:S\ra T$ with  $S\in\Tilde\sigma_{v,\eps}$
and $T$ arbitrary is compact. Indeed, the bound on the covolume bounds the length functions of $S$ and $T$ from above, and it suffices to check that the length function
of $T$ cannot accumulate to $0$. Now there is a finite set $F\subset G$ of non-peripheral elements such that for all $f\in \Opt_{\Tilde \sigma_{v,\eps}}$, there exists some element $g\in F$ whose axis in $S$ is legal (see the proof of \cite[Theorem~4.7]{Hor3}), so $l_T(g)=l_S(g)$ is bounded from below by the systole of $S$ which is at least $\eps$.

Let us now prove that $Z_\tau$ is Borel.
Since there are countably many simplices, it suffices to prove that for each $n\in \bbN$ and every $v,\eps>0$, the set
$$Z'=Z_\tau\cap \left( [-n,n]\times \Opt_{\Tilde \sigma_{v,\eps}\ra\AT}\right)= \left\{(t,f)\in [-n,n]\times \Opt_{\Tilde \sigma_{v,\eps}\ra \AT}\,|\ \tau\in \cals_t(f)\right\}$$ 
is Borel. Let $\tilde \tau\subset \calo$ be the set of trees whose projective class lies in $\tau$.
Given $\delta>0$, let $X_{\tau,\delta}$ be the subset of $[-n,n]\times \Opt_{\Tilde \sigma_{v,\eps}}\times \Opt_{\Tilde \sigma_{v,\eps}\ra \Tilde\tau} \times \Opt_{\Tilde \tau}$ made of all tuples $(t,f,g,h)$ such that  
$f=h\circ g$, the covolume of the source of $h$ equals $e^{-t}$, and the systole of the source of $h$ is equal to or larger than $\delta$. 
We note that $f,g,h$ belong to the compact spaces $\Opt_{\Tilde \sigma_{v,\eps}}$, $\Opt_{\Tilde \sigma_{v,\eps}\ra \Tilde\tau}$ and $\Opt_{\Tilde\tau_{e^{-t},\delta}}$, respectively. Since composition is continuous, and since the covolume and systole are continuous functions on $\Tilde \tau$, the set $X_{\tau,\delta}$ is compact. Let $X'_{\tau,\delta}$ be the projection of $X_{\tau,\delta}$ to the first two coordinates. Then  $$Z'=\bigcup_{k\in\mathbb{N}}X'_{\tau,\frac{1}{k}}\cap ([-n,n]\times \Opt_{\ra\AT}),$$ and therefore $Z'$ is Borel.  
\end{proof}

\subsection{Measurability considerations}\label{sec-measurability}

This technical section gives tools to prove some measurability
properties. It can be omitted in a first reading if the reader wishes
to ignore all measurability considerations. 

\subsubsection{Measurable enumeration of directions in trees}\label{sec-dir}

Let $\{(g,h)_i\}_{i\in\mathbb{N}}$ be an enumeration of $G^2$. Given $T\in\baro$, we say that $(g,h)\in G^2$ is a \emph{disjoint pair} for $T$ if $g$ and $h$ are both hyperbolic in $T$, and their axes are disjoint. Given $n\in\mathbb{N}$ and $T\in \baro$, we let  $(g_n(T),h_n(T))$ be the $n^{\text{th}}$ pair of elements in the above enumeration that is a disjoint pair for $T$.

\begin{lemma}\label{B-mes}
For all $n\in\mathbb{N}$ and all $(g,h)\in G^2$, the set $$B_{n,g,h}:=\{T\in\baro|(g_n(T),h_n(T))=(g,h)\}$$ is a Borel subset of $\baro$.
\end{lemma}

\begin{proof}
It is well-known (see \cite[1.8]{CM}) that $(g,h)$ is a disjoint pair for $T$ if and only if $||g||_T,||h||_T>0$ and
$||gh||_T > ||g||_T + ||h||_T$, an open condition. The fact that $(g,h)$ is the   $n^{\text{th}}$ disjoint pair in the enumeration can be expressed a   Boolean combination of such open sets. Therefore $B_{n,g,h}$ is a Borel set.
\end{proof}

Given $g\in G$ and $T\in\baro$ such that $g$ is hyperbolic in $T$, we denote by $C_g(T)$ the axis of $g$ in $T$. Given a disjoint pair $(g,h)\in G^2$ for $T$, we define $v_{(g,h)}(T)$ as the endpoint in $C_{g}(T)$ of the bridge joining $C_{g}(T)$ to $C_{h}(T)$.
We also define $d_{(g,h)}(T)$ as the branch direction at $v_{(g,h)}(T)$ pointing towards $C_{h}(T)$. Given $n\in\mathbb{N}$, we then let $v_n(T):=v_{(g_n(T),h_n(T))}(T)$ and $d_n(T):=d_{(g_n(T),h_n(T))}(T)$, and we let $v_n'(T):=v_{(h_n(T),g_n(T))}(T)$. 
In particular $d_n(T)$ is the direction based at $v_n(T)$ pointing towards $v'_n(T)$.
Notice that for any branch direction $d$ in $T$, there exists $n\in \bbN$ such that $d=d_{n}(T)$. 
Similarly, any branch point arises as $v_n(T)$ for some $n\in\mathbb{N}$, see \cite{Pau}. 

\begin{lemma}\label{dist-mes}
  For all $n,m\in\bbN$, and $g\in G$, the maps $T\mapsto d_T(v_n(T),gv_m(T))$ and
 $T\mapsto d_T(v_n(T),C_g(T))$ are Borel.
\end{lemma}

\begin{proof}
Let $X_n(T):=\{g_n(T), g_n(T)h_n(T), h_n(T)g_n(T)\}$. The branch point $v_n(T)$ can be defined as the unique point in the intersection of the axes of the elements in $X_n(T)$.
It follows that $$d_T(v_n(T),gv_m(T))=\max \{d_T(C_\alpha(T),C_\beta(T))|\alpha\in X_n(T),\beta\in X_m(T)^g\}$$ and that $$d_T(v_n(T),C_g(T))=\max \{d_T(C_\alpha(T),C_g(T))|\alpha\in X_n(T)\}.$$ 
Since $d_T(C_\alpha(T),C_\beta(T))=\frac{1}{2}\max\{0, ||\alpha\beta||_T-||\alpha||_T-||\beta||_T\}$, the lemma follows.
\end{proof}

\begin{lemma}\label{dir-mes}
For all $n,m,p,k,l\in\mathbb{N}$ and all $g\in G$, the following sets are Borel subsets of $\baro$:
\begin{enumerate}
\item $\{T\in\baro|v_n(T)=gv_m(T)\},$
\item $\{T\in\baro|v_n(T),v_m(T),v_p(T)\text{~are pairwise distinct and  aligned in this order}\},$ 
\item $\{T\in \baro| d_n(T)\text{ points towards }g.v_m(T)\},$ 
\item $\{T\in\baro|d_n(T)=gd_m(T)\},$
\item $\{T\in\baro|g(d_n(T),d_m(T))\text{~is a turn lying in the segment~} [v_k(T),v_l(T)]\}.$
\end{enumerate}
\end{lemma}

\begin{proof}
The first two assertions follow from Lemma \ref{dist-mes}. For the third one, we note that $d_n(T)$ fails to point towards $gv_m(T)$ if and only if either $gv_m(T)=v_n(T)$, or else the points   $v'_n(T),v_n(T),gv_m(T)$ are pairwise distinct and aligned in this order. The fourth assertion follows from the third because 
$d_n(T)=gd_m(T)$ if and only if $v_n(T)=gv_m(T)$ and $d_n(T)$ points towards $gv'_m(T)$.
The last assertion also follows   as $g(d_n(T),d_m(T))$ is a turn lying in the segment $[v_k(T),v_l(T)]$ if and only if $v_n(T)=v_m(T)$, and either 
\begin{itemize}
\item $d_n(T)$ points towards $g^{-1}v_k(T)$, and $d_m(T)$ points towards $g^{-1}v_l(T)$, or
\item $d_m(T)$ points towards $g^{-1}v_k(T)$, and $d_n(T)$ points towards $g^{-1}v_l(T)$.\qed
\end{itemize}
\renewcommand{\qedsymbol}{}
\end{proof}

\subsubsection{Enumerating morphisms}\label{sec-enum-morphisms}

For every open simplex $\Delta\in\Simp$, every tree in $\Tilde\Delta$ is obtained by assigning positive lengths to the edges of some combinatorial tree $S_\Delta$. 
Choose a collection $v_1^\Delta,\dots, v^\Delta_{k(\Delta)}$  
of representatives of the orbits of vertices 
with trivial stabilizer in $S_\Delta$.
Given any tree $S$ in the closure of $\tilde \Delta$ in $\calo$,
we still denote by $v_i^\Delta$ the corresponding vertex of $S$ (there may exist $i\neq j$
such that $v_i^\Delta$ and $v_j^\Delta$ are in the same orbit when viewed in $S$).

We define a \emph{decorated simplex} as a pair $(\Delta,\theta)$, where $\Delta\in \Simp$, and where $\theta:\{1,\dots,k(\Delta)\}\to\mathbb{N}$ is a map. We denote by $\Simp^\ast$ the countable collection of all decorated simplices. Given a tree $T\in\baro$ and a decorated simplex $(\Delta,\theta)$, there is a 
unique  metric tree $S$ in the closure of $\tilde \Delta$,
and a unique $G$-equivariant morphism 
$f:S\ra T$ which is isometric on edges of $S$, and maps the vertex $v_i^\Delta\in S$ to the branch point $v_{\theta(i)}(T)$ in $T$ for all $i$.
We denote this unique morphism $f$ by $f_{(\Delta,\theta),T}$.

\begin{lemma}\label{opt-borel}
The map
\begin{displaymath}
\begin{array}{cccc}
\Simp^\ast\times\baro &\to &\Mor\\
((\Delta,\theta),T)&\mapsto & f_{(\Delta,\theta),T}
\end{array}
\end{displaymath}  
\noindent is Borel. Its image consists of all morphisms that are isometric on edges, and send vertices   with trivial stabilizer to branch points.
\end{lemma}

  Notice that any vertex $v$ of $S$ with nontrivial stabilizer has to be sent to the unique generalized branch point of $T$ fixed by $G_v$. 
 
\begin{proof}
The last statement in the lemma follows from the observation that every branch point in $T$ arises as $v_n(T)$ for some $n\in\mathbb{N}$, and a morphism which is isometric on edges is completely determined by the images of the vertices $v_1^\Delta,\dots,v_{k(\Delta)}^\Delta$   (the image of a vertex $v$ with nontrivial stabilizer $G_v$ has to be equal to the unique vertex of $T$ fixed by $G_v$). To show that the map in the lemma is Borel, it is enough to show that for each decorated simplex $(\Delta,\theta)$, the map  
\begin{displaymath}
\begin{array}{cccc}
\baro &\to &\Mor\\
T &\mapsto & f_{(\Delta,\theta),T}
\end{array}
\end{displaymath}
\noindent is Borel. This is because it is continuous when restricted to each of the countably many fibers of the map 
\begin{displaymath}
\begin{array}{cccc}
\baro &\to & (G^2)^{k(\Delta)}\\
T & \mapsto & (g_{\theta(n)}(T),h_{\theta(n)}(T))_{1\le n\le k(\Delta)}
\end{array}
\end{displaymath}
\noindent and these fibers are Borel in view of Lemma~\ref{dir-mes}.
\end{proof}

\begin{rk}\label{rk_enum_morphisms}
  Given an enumeration of the set of decorated simplices, we obtain for each $T\in \baro$ an enumeration 
of all morphisms $F_{n,T}:S_{n,T}\ra T$ from all simplicial trees to $T$ that send vertices to branch points, and are isometric on edges. 
Moreover, the maps $T\mapsto S_{n,T}$ and $T\mapsto F_{n,T}$ are Borel for every $n$.
\end{rk}

Given $S\in\calo$ and $T\in\baro$, we say that a morphism $f:S\to T$ is \emph{very optimal} if $f$ is optimal and sends vertices of $S$
with trivial stabilizer to branch points in $T$. We denote by $\VOpt$ the subspace of $\Opt$ made of very optimal morphisms.

\begin{rk}\label{rk_exist_vopt}
  We note for future use that for every tree $T\in \baro$, there exists $S\in\calo$ and a very optimal morphism $f:S\ra T$.
  Indeed, if $\calf\neq \es$, one may take for $S$ a tree such that all vertices have non-trivial stabilizer, with the unique metric
  so that the unique equivariant map $f:S\ra T$ which is linear on edges is a morphism, necessarily very optimal.
  If $(G,\calf)=(F_N,\es)$, one may take $S$ homeomorphic to the Cayley graph of $F_N$ with respect to a free basis $\{a_1,\dots,a_N\}$,
and choose $f$ sending the base point of $S$ to a branch point in the axis of $a_1$ in $T$.
\end{rk}

\begin{lemma}\label{vopt-borel}
The preimage of $\VOpt$ in $\Simp^\ast\times\baro$  under the map of Lemma \ref{opt-borel}, is Borel.
\end{lemma}

\begin{proof}
It is a consequence of Lemma~\ref{dir-mes} that given a finite set $E$ of edges of the quotient graph $S_\Delta/G$, the collection of all $((\Delta,\theta),T)\in\Simp^\ast\times\baro$ such that the edges projecting to $E$ are collapsed in the source of $f_{(\Delta,\theta),T}$ is Borel. By restricting to each of these finitely many subsets, we can assume that no edge in $S_\Delta$ is collapsed. Then $f_{(\Delta,\theta),T}$ is very optimal if and only if for every $n\in\{1,\dots,k(\Delta)\}$, there exist $p,q\in\{1,\dots,k(\Delta)\}$ and $g,h\in G$ such that $gv_p^\Delta$ and $hv_q^\Delta$ are adjacent to $v_n^\Delta$ in $\Delta$, and the points $gv_{\theta(p)}(T),v_{\theta(n)}(T),hv_{\theta(q)}(T)$ are aligned in this order in $T$. The conclusion thus follows from Lemma~\ref{dir-mes}.
\end{proof}

\subsubsection{A $\sigma$-algebra on the set of turning classes}\label{sec_algebra}

We denote by $\oturn$ the collection of all turning classes on trees in   $\overline{\mathcal{O}}$.
There is an embedding   $\Phi:\oturn\hookrightarrow \overline{\calo}\times \{0,1\}^{\bbN^2}$
sending $(T,\calt)$ to   the couple consisting of the tree $T$ and of the map taking value $1$ on $\{(n,m)\in \bbN^2| (d_n(T),d_m(T))\in\calt\}$.
This defines a $\sigma$-algebra on $\oturn$ by pulling back the standard Borel $\sigma$-algebra.

The group $\Out(G,\calf)$ has a natural action on $\oturn$, defined as follows. Let $\calb$ be a turning class   on a tree $T\in\baro$, and let $\Phi\in\Out(G,\calf)$. Then the tree $\Phi.T$ is the same metric space as $T$ (equipped with a twisted $G$-action), and we let $\Phi.\calb$ be the turning class on $\Phi.T$ consisting of the same turns as $\calb$. 
One easily checks that this action is measurable.

\subsection{The case of free actions}\label{sec_FN}

In this subsection, we consider the case free arational actions of the free group,
which occurs in the case where  $(G,\calf)=(F_N,\es)$.
This can be seen as a warm up for the more general situation that will be tackled in the next subsection.

In this case, $\calo$ is the original Culler--Vogtmann  Outer space, consisting of free actions of $F_N$ on simplicial trees.
Moreover, any arational action $T\in\AT$ is either a free action of $F_N$, or $T$ is dual
to an  arational measured foliation on a surface with exactly one boundary component and with fundamental group $F_N$ (\cite{Rey}, see Section \ref{sec-arat-back}).
We denote by $\ATf$ the subset of free arational actions of $(F_N,\es)$.
The interesting case is when $N\geq 3$ since $\ATf$ is empty otherwise.

Because of the freeness of the action, for any $T\in\ATf$, there are only finitely many directions at each branch point
(not just finitely many \emph{orbits} of directions).
In particular, there are only finitely many turns at each branch point, hence
finitely many orbits of turns.
This is the main difference with the general case. 

We now prove the following result, which is a particular case of the more technical Proposition \ref{name-enough}.

\begin{theo}\label{thm:FN-case}
  The action of $\Out(F_N)$ on $\bbP\ATf$ is Borel amenable.
\end{theo}

\begin{proof}
Since $(G,\calf)=(F_N,\es)$, each simplex of $\bbP\calo$
 has finite stabilizer. Thus by Proposition \ref{amen}, it suffices
  to define a sequence of Borel maps $\mu_n$ associating to $T\in\bbP\ATf$
  a probability measure $\mu_n^T\in\Prob(\Simp)$ such that for all $\Phi\in\Out(F_N)$
  and all $T\in \bbP\ATf$, one has $$||\Phi.\mu_n^T-\mu_n^{\Phi.T} ||_1\xrightarrow{n\ra\infty} 0.$$
  
  We fix a continuous section $\alpha:\mathbb{P}\baro\to\baro$.
  Given $T\in\ATf$, let $k(T)$ be the number of turning classes of all very optimal
  morphisms $f:S\ra T$,  with $S$ varying in $\calo$. This number is finite because being a free action, $T$ has only finitely many orbits of turns.
  It is non-zero because very optimal morphisms exist (Remark \ref{rk_exist_vopt}).

  Let $f_1^T,\dots,f_{k(T)}^T$ be a collection of very optimal morphisms $f_i^T:S_i^T\ra \alpha(T)$
representing all these turning classes.
We choose this collection that is smallest for the lexicographic order, relative to a measurable enumeration of morphisms with range $T$ 
as in Remark \ref{rk_enum_morphisms}.

We now consider the sequence of probability measures $\mu_n(f)$ associated to an optimal morphism $f$ with arational target
as in Section \ref{sec-finite-width}, and we define
 $$\mu_n^T:=\frac{1}{k(T)}\sum_{i=1}^{k(T)}\mu_n(f_{i}^T).$$  
Applying $\Phi\m$ to the morphism $f_i^{\Phi.T}:S_i^{\Phi.T}\ra \alpha(\Phi.T)$,
we get a morphism $f'_i: S'_i\ra \lambda \alpha(T)$ with $S'_i=\Phi\m. S_i^{\Phi.T}$, and $\lambda>0$ such that $\Phi\m\alpha(\Phi.T)=\lambda \alpha(T)$.
The morphism $f'_i$ has the same turning class as $f_{\sigma(i)}^T$ for some permutation $\sigma$ of $\{1,\dots,k(T)\}$.
Since the two target trees are homothetic, Corollary~\ref{cor-shift} implies that for all $i\in\{1,\dots,k(T)\}$, we have
$$||\mu_n(f^T_{\sigma(i)})-\mu_n(f'_i)||_1\to 0$$ as $n$ goes to $+\infty$. Averaging over $i\in \{1,\dots,k(T)\}$, it then follows that 
$$||\mu_n^T-\Phi^{-1}.\mu_n^{\Phi.T}||_1\to 0$$ as $n$ goes to $+\infty$.   

The measurability of the maps  $T\mapsto\mu_n^T$ is easy, we refer to the proof of Proposition~\ref{name-enough}
for details.
\end{proof}

\subsection{Namable turning classes}

 The notion of namable turning class defined below is a technical notion
that allows us to extend the argument given in the previous subsection in the case where $(G,\calf)=(F_N,\es)$.

Recall that $\oturn$ denotes the collection of all turning classes on trees in $\mathcal{O}$. More generally, given $X\subseteq\baro$, we denote by $\turn{X}$ the collection of all turning classes on trees in $X$. Given a turning class $\calb$ on a tree $T\in\baro$, and $\lambda>0$, we denote by $\lambda\calb$ the turning class on $\lambda T$ consisting of the same turns.

\begin{de}[Having namable turning classes]\label{de-namable}
An $\Out(G,\calf)$-invariant Borel subset $X\subseteq\baro$ \emph{has namable turning classes} if there exists a map $$\Name:\turn{X}\to\mathbb{N}\cup\{\perp\}$$ such that 
\begin{enumerate}
\item $\Name$ is measurable,
\item for all $\lambda\in\mathbb{R}_+^\ast$ and all turning classes $\calb\in\turn{X}$, we have $\Name(\lambda\calb)=\Name(\calb)$,
\item $\Name$ is $\text{Out}(G,\calf)$-invariant, i.e.\ $\Name(\Phi.\calb)=\Name(\calb)$ for all $\calb\in\turn{X}$ and all $\Phi\in\text{Out}(G,\calf)$,
\item for all $T\in X$ and all $n\in\mathbb{N}$, there are only finitely many turning classes on $T$ whose name is $n$, and
\item for all $T\in X$, there exists a very optimal morphism with range $T$ whose turning class 
 has a name different from $\perp$.
\end{enumerate} 
\end{de}

\begin{rk}
  In the case where $(G,\calf)=(F_N,\es)$, then the set $X=\ATf$ of free arational trees has namable turning classes.
  Indeed, define $\Name(T)=0$ for all $T\in\ATf$. Assertion 4 is a consequence of the fact
  that being free, $T$ has only finitely many turning classes. Assertion 5 follows from Remark \ref{rk_exist_vopt}.
\end{rk}

\begin{prop}\label{name-enough}
  Let $X\subseteq\AT$ 
  be an $\Out(G,\calf)$-invariant Borel subset which has namable turning classes.
\\ Then there exists a sequence of Borel maps 
\begin{displaymath}
\begin{array}{cccc}
\mathbb{P}X&\to &\mathrm{Prob}(\Simp)\\
T &\mapsto & \mu_n^T
\end{array}
\end{displaymath}
\noindent such that for all $\Phi\in\text{Out}(G,\calf)$ and all $T\in \mathbb{P}X$, one has $$||\Phi.\mu_n^T-\mu_n^{\Phi.T}||_1\to 0$$ as $n$ goes to $+\infty$.
\end{prop}

\begin{proof}
Fix a continuous section $\alpha:\mathbb{P}\baro\to\baro$. To any $T\in \mathbb{P}X$, we are going to associate a finite set of very optimal morphisms $f_i^T$ 
with range $\alpha(T)$ as follows.
Let $i_0(T)$ be the smallest integer $i\in\mathbb{N}$ for which there exists a very optimal morphism with range $\alpha(T)$, whose turning class with respect to $T$  
has name $i$: this exists in view of the last hypothesis from Definition~\ref{de-namable}. 
Let $f_1^T,\dots,f_{k(T)}^T$ be a collection of very optimal morphisms $f_i^T:S_i^T\ra \alpha(T)$  such that 
the turning classes of $f_1^T,\dots,f_{k(T)}^T$ are exactly the turning classes named $i_0(T)$ (without repetition).
We choose this collection that is smallest for the lexicographic order, relative to a measurable enumeration of morphisms with range $T$ 
as in Remark \ref{rk_enum_morphisms}.

Consider the sequence of probability measures $\mu_n(f)$ associated to an optimal morphism $f$ with arational target
as in Section~\ref{sec-finite-width}.
We claim that the probability measures
$$\mu_n^T:=\frac{1}{k(T)}\sum_{i=1}^{k(T)}\mu_n( f_{i}^T)$$
satisfy the desired conclusion.

Notice that $\Out(G,\calf)$-invariance of $\Name$ implies that for all $T\in \mathbb{P}X$ and all $\Phi\in\Out(G,\calf)$, we have $i_0(\Phi.T)=i_0(T)$ and $k(\Phi.T)=k(T)$.
Applying $\Phi\m$ to the morphism $f_i^{\Phi.T}:S_i^{\Phi.T}\ra \alpha(\Phi.T)$,
we get a morphism $f'_i: S'_i\ra \lambda \alpha(T)$ with $S'_i=\Phi\m. S_i^{\Phi.T}$, and $\lambda>0$ such that $\Phi\m\alpha(\Phi.T)=\lambda \alpha(T)$.
The morphism $f'_i$ has the same turning class as $f_{\sigma(i)}^T$ for some permutation $\sigma$ of $\{1,\dots,k(T)\}$.
Since the two target trees are homothetic, Corollary~\ref{cor-shift} implies that
for all $i\in\{1,\dots,k(T)\}$, we have
$$||\mu_n(f^T_{\sigma(i)})-\mu_n(f'_i)||_1\to 0$$ as $n$ goes to $+\infty$. Averaging over $i\in \{1,\dots,k(T)\}$, it then follows that 
$$||\mu_n^T-\Phi^{-1}.\mu_n^{\Phi.T}||_1\to 0$$ as $n$ goes to $+\infty$.   

We finally check that the maps $T\mapsto\mu_n^T$ are Borel. 
Let $F_{n,T}$ be a Borel enumeration of morphisms as in Remark \ref{rk_enum_morphisms}.
Then $$i_0(T)=\min\{i\in\mathbb{N}|\exists n\in\bbN,\  F_{n,\alpha(T)} \text{ is optimal and its turning class has name~}i\}.$$  
Since optimality is a Borel condition (Lemma \ref{vopt-borel}) and 
the map $\Name$ is Borel, we deduce that the map $T\mapsto i_0(T)$ is Borel. Similarly, the maps $T\mapsto k(T)$ and $T\mapsto f_i^T$ are Borel. In view of Lemma~\ref{sphere-borel}, we thus get that $T\mapsto\mu_n^T$ is Borel.
\end{proof}

\section{Arational trees have namable turning classes.}\label{sec-name}

The goal of this section is the following result (see Definition \ref{de-namable} for definitions).

\begin{prop}\label{prop_namable}
The set $\AT$ of arational trees has namable turning classes.
\end{prop}

We will prove separately that the sets $\Geom$ and $\NonG$ made of geometric and nongeometric arational trees have namable turning classes. Thus Proposition \ref{prop_namable}
is a consequence of Corollary \ref{cor-nongeom} and Proposition \ref{prop_nom-geom}.

We start with the following observation.

\begin{lemma}\label{mes-geom}
The sets $\mathcal{G}eom$ and $\mathcal{NG}eom$ are Borel subsets of $\AT$.
\end{lemma}

\begin{proof}
It suffices to show that $\mathcal{NG}eom$ is a Borel subset of $\AT$. Let $\{g_i\}_{i\in\mathbb{N}}$ be an enumeration of $G$. 

Using the third characterization of non-geometric trees given in Lemma~\ref{strong-cv}, a tree $T\in \AT$ is non-geometric if and only if 
for every finite subset $F\subset G$, there exists $S\in \calo$ with a morphism $S\ra T$
such that the length functions of $T$ and $S$ agree on $F$.
Given a simplex $\Delta\in \Simp$, denote by $\Tilde \Delta_{\bbQ}$ the set of trees in $\tilde \Delta$ with rational edge lengths.
It clearly follows that $T\in\AT$ is non geometric if and only if
for all finite subsets $F\subset G$, there exists a simplex $\Delta\in \Simp$ such that for all $k\geq 1$,
there exists $S\in \tilde \Delta_\bbQ$ such that 
\begin{itemize}
\item $\forall g\in G,$ $||g||_S\geq ||g||_T$ and
\item $\forall g\in F$, $||g||_S \leq ||g||_T + \frac1k$.
\end{itemize}
This expresses $\mathcal{NG}eom$ as countable intersections and unions of Borel sets.
\end{proof}

\subsection{Non-geometric arational trees}

Given a tree $T\in\baro$, the \emph{full turning class} in $T$ is the turning class consisting of all turns at all branch points in $T$.

\begin{prop}\label{nongeom}
For all $T\in\mathcal{NG}eom$, there exists a tree $S\in\calo$ and a very optimal morphism $f:S\to T$ whose turning class with respect to $T$ is the full turning class of $T$. 
\end{prop}

\begin{proof}
Since $T$ is nongeometric, it is not an arational surface tree, and therefore it is relatively free. Let $X\subseteq T$ be a finite subtree that contains 
\begin{itemize}
\item a set of representatives $\mathcal{V}=\{v_1,\dots,v_l\}$ of the orbits of branch points of $T$,
\item for each branch point $v_i\in\mathcal{V}$, a set of representatives of the orbits of directions at $v_i$.
\end{itemize}
Since $T$ is nongeometric, Lemma~\ref{strong-cv} ensures that there exists a nonstationary direct system of trees $(S_n)_{n\in\mathbb{N}}\in\calo^\mathbb{N}$ that converges strongly to $T$, coming with morphisms $f_n:S_n\to T$ and $f_{nm}:S_n\to S_m$ for all $m\ge n$ (these morphisms can be chosen optimal). Let $X_0\subseteq S_0$ be a finite subtree of $S_0$ whose $f_0$-image contains $X$, and such that for every $v_i\in\mathcal{V}$ with nontrivial stabilizer in $T$, the vertex of $S_0$ with stabilizer $G_{v_i}$ is contained in $X_0$. By definition of strong convergence, there exists $n\in\mathbb{N}$ such that $f_n$ is isometric when restricted to $X_n:=f_{0n}(X_0)$. In particular $X_n$ contains an isometric copy of $X$, and the preimage in $X_n$ of every vertex $v$ of $X$ with 
non-trivial stabilizer 
in $T$, is the vertex of $S_n$ having the same peripheral point stabilizer $G_v$. Let now $(d,d')$ be a turn based at a vertex $v_i$ in $T$ (up to translating we can assume that $v_i\in\mathcal{V}$). If $v$ has trivial stabilizer in $T$, then the assumptions made on $X$ ensure that $(d,d')$ lifts to $S_n$. If $v$ has non-trivial stabilizer in $T$, then there exist $g,g'\in G_v$ such that $(gd,g'd')$ lifts at the vertex of $S_n$ with stabilizer $G_v$, and therefore $(d,d')$ also lifts. Therefore, every turn in $T$ lifts to a turn in $S$, so the turning class of $f_n$ is the full turning class of $T$. Notice that by slightly folding if needed, we can always arrange that $f_n$ maps vertices of $S_n$ to branch points of $T$, i.e.\ $f_n$ is very optimal.
\end{proof}

\begin{cor}\label{cor-nongeom}
The subset $\NonG\subseteq \baro$ has namable turning classes.
\end{cor}

\begin{proof}
In view of Proposition~\ref{nongeom}, the map $\Name$ that sends the full turning class of any tree $T\in \NonG$ to  $0$, and any other turning class to $\perp$, satisfies the assumptions in Definition~\ref{de-namable}.
\end{proof}

\subsection{Geometric arational trees}

The goal of the present section is to prove the following fact.

\begin{prop}\label{prop_nom-geom}
The subset $\mathcal{G}eom\subseteq\baro$ has namable turning classes.
\end{prop}

 We treat separately the cases of surface arational trees and of relatively free arational trees.
Thus Proposition \ref{prop_nom-geom} is a consequence of Propositions \ref{sat-name} and
\ref{name-dec}.

\subsubsection{Ubiquitous turns}

\begin{de}[Ubiquitous turns]
Let $T\in\baro$. A turn $U$ in $T$ is \emph{ubiquitous} if for every interval $I$, there exists $g\in G$ such that $gU$ is contained in $I$. 
\end{de}

Let $T\in\baro$ be a geometric tree, and let $G\actson\Sigma$ be a
band complex as in Section~\ref{sec-geom-back} to which $T$ is
dual. We denote by $V$ the set of \emph{vertices} of $\Sigma$,
i.e.\ points in $\Sigma$ belonging to a base tree $K_v$ which are
either extremal in one of the bands they belong to, or are at least
trivalent in $K_v$. Recall that all leaves of $\Sigma$ are trees. A leaf of $\Sigma$ is \emph{singular} if it
contains a vertex of $\Sigma$. Notice that if $x\in T$ is a branch point, then the leaf of $\Sigma$ that projects to $x$ in $T$ is
singular. Given a singular leaf $l$ of $\Sigma$ projecting to a point
$x\in T$, we denote by $l^0\subseteq l$ the union of the convex
hull
of $l\cap V$ in $l$ together with all finite connected components of $l\setminus V$. Notice that since the set $V$ is $G$-finite, the intersection $l\cap V$ is $G_l$-finite (where $G_l$ denotes the stabilizer of the leaf $l$). Since there are only finitely many
$G_l$-orbits of finite components in $l\setminus V$, we deduce that $l^0/G_l$ is compact (it is a finite graph). We call a turn at $x$ in $T$ \emph{regular} if it comes from a turn in a base tree at a point in $l\setminus l^0$, and \emph{singular} otherwise.

\begin{lemma}\label{ubi-0}
Let $T$ be a geometric tree which is mixing (i.e.\ dual to a band complex with a single orbit of minimal components), and let $x\in T$ be a branch point.
\\ A turn at $x$ is ubiquitous if and only if it is regular.
\\ Every segment in $T$ contains only finitely many singular turns based at a point in $G.x$.
\end{lemma}

\begin{proof}
Let $l\subset \Sigma$ be the leaf representing $x$.
Since $l^0$ is $G_l$-cocompact, for each band $B$, the orbit $G.l^0$ intersects $B$ in a finite set of leaf segments (otherwise $B$ would have nontrivial stabilizer). 
This shows  that singular turns at $x$ are non-ubiquitous. Moreover, it implies that for each segment $\Tilde I\subset K_v$ contained
in a band, only finitely many turns in the projection $I$ of $\tilde I$ to $T$ are singular turns based at a point in $G.x$. 
This obviously also holds if the interior of $\tilde I$ intersects no band (which in fact cannot happen since $T$ is mixing).
Since every arc in $T$ is contained in a finite union of such segments $I$, 
it also follows that every arc in $T$ contains only finitely many singular turns based at a point in  $G.x$.  

We will now prove that every regular turn $U$ at $x$ is ubiquitous. Since $U$ is regular, there exist a base tree $K_v$ intersecting $l\setminus l^0$ and a point $z\in K_v\cap (l\setminus l^0)$ having the following property: denoting by $l_z\subset l$ the connected component of $l\setminus l^0$ containing $z$, for every transverse interval $I\subset \Sigma$ intersecting $l_z$, the two directions at the point $I\cap l_z$ define the turn $U$.
Since the band complex $\Sigma/G$ has a unique minimal component, by \cite[Theorem~3.1]{GLP} the orbit of $l_z$ intersects each transverse tree $K_v\subset \Sigma$ in a dense set.  Since every arc in $T$ is a finite concatenation of arcs that are the image of intervals contained in base trees, it follows that $U$ is ubiquitous. This completes the proof of the lemma.
\end{proof}

Using the fact that there are only finitely many $G$-orbits of singular leaves and $l^0/G_l$ is a finite graph, we get the following as a consequence of Lemma~\ref{ubi-0}.
 
\begin{cor}\label{ubi}
Let $T$ be a geometric tree which is mixing (i.e.\ dual to a band complex with a single orbit of minimal components).
\\ Then there are only finitely many orbits of ubiquitous turns at branch points.
\\ Every interval contains only finitely many non-ubiquitous turns.
\qed
\end{cor}

Given a branch point $x\in T$ and a collection $\mathcal{T}$ of turns at $x$, the \emph{Whitehead graph} at $x$ defined by $\mathcal{T}$ is the graph having one vertex for each direction at $x$, and an edge between two directions if the turn they form belongs to $\mathcal{T}$.

\begin{lemma}\label{lem_whitehead}
Let $T\in\mathcal{G}eom$. Let $x\in T$ be a branch point. If the Whitehead graph at $x$ defined by the collection of all ubiquitous turns is disconnected, then $T$ is not indecomposable.
\end{lemma}

\begin{proof}
Since the Whitehead graph defined by ubiquitous turns at $x$ is not connected, we have an equivariant partition of the directions at $x$ such that every ubiquitous 
turn joins two equivalent directions. We say that a turn at a point in the orbit of $x$ is \emph{allowed} if it is a translate of a
turn joining two equivalent directions at $x$.
We now construct a transverse covering of $T$: say that a subtree $Y$ is \emph{allowed} if all its turns at points in the orbit of $x$
are allowed. Let $\caly$ be  the set of all nondegenerate maximal allowed subtrees. The last assertion from Corollary~\ref{ubi} implies that $\caly$ is nonempty (and actually covers $T$): any arc $I$ can be covered by finitely many subarcs that only contain ubiquitous turns. In addition, is not the trivial covering (i.e.\ $\caly\neq\{ T\}$) because there are at least two equivalence classes at $x$.
 
We claim that the family $\caly$ is transverse;  this will show that $T$ is not indecomposable and complete the proof of the lemma. Otherwise, we can find $Y,Y'\in\caly$ two distinct subtrees with nondegenerate intersection. 
By maximality, there exists a turn $(d,d')$ in $Y\cup Y'$ which is not allowed. Since $Y$ and $Y'$ are allowed, up to exchanging $d$ and $d'$ 
we can assume that $d\subset Y\setminus Y'$ and $d'\subset Y'\setminus Y$. Let $b\in Y\cap Y'$ be the common base point of these directions, and let $d''$ be a direction at $b$ in $Y\cap Y'$: this exists because $Y\cap Y'$ is nondegenerate. Then $(d,d'')$ and $(d',d'')$ are allowed, hence so is $(d,d')$, a contradiction.
\end{proof}

We finish this section by proving measurability of the ubiquity property (we recall the definition of the directions $d_n(T)$ from Section~\ref{sec-dir}).

\begin{lemma}\label{ubi-mes}
For all $n,m\in\mathbb{N}$, the set $$U_{nm}:=\{T\in\baro|(d_n(T),d_m(T))\text{~is a ubiquitous turn}\}$$ is a Borel subset of $\baro$.
\end{lemma}

\begin{proof}
  By Lemma~\ref{dir-mes}, the set
 $$D_{k,l,m,n,g}:=\left\{T\in\baro\ \middle |
\begin{array}{l}
\bullet\ \text{ either } v_k(T)=v_l(T), \text{ or } \\
\bullet\ g(d_n(T),d_m(T))\text{~is a turn lying in the segment~} [v_k(T),v_l(T)]
\end{array}
\right\}$$ is a Borel subset of $\baro$. Since $$U_{nm}=\bigcap_{k,l\in\mathbb{N}}\bigcup_{g\in G}D_{k,l,m,n,g},$$ the result follows.
\end{proof}

\subsubsection{Surface arational trees}\label{sec_SAT}

\newcommand{\SAT}{\mathcal{SAT}}

We denote by $\SAT$ the subspace of $\AT$ made of surface arational trees. We recall from Section~\ref{sec-arat-back} that this coincides with the subspace of $\AT$ made of non relatively free trees: in other words, a tree $T\in\AT$ is in $\SAT$ if and only if there exists a nonperipheral element $g\in G$ such that $||g||_T=0$. This shows that $\SAT$ is a Borel subset of $\AT$.

\begin{figure}[htb!]
\begin{center}
\includegraphics[width=.5\textwidth]{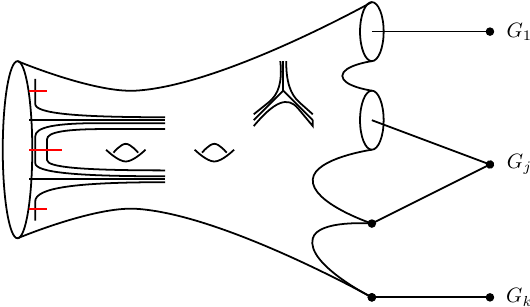}
\caption{The tree $T$ from the proof of Proposition~\ref{sat-name} is the arational surface tree dual to the foliation on the orbifold. The tree $S_t$ is obtained by cutting a segment of size $t$ along each of the red slits, and considering the simplicial tree dual to the foliation obtained in the cut surface.}
\label{fig-slit}
\end{center}
\end{figure}

\begin{prop}\label{sat-name}
The subset $\SAT$ has namable turning classes.
\end{prop}

\begin{proof}
Let $T\in\SAT$. The tree $T$ splits as a graph of actions where the only nontrivial vertex action is dual to a foliation on a $2$-orbifold with a single unused boundary component (see Section \ref{sec-arat-back}). Up performing Whitehead moves on the foliation 
(which does not change the dual tree), we can assume that half-leaves starting at the unused boundary curve do not contain singularities of the foliation, and that all singularities on the boundary are $3$-pronged. By putting slits of sufficiently small length $t>0$ between any two half-leaves of the foliation defining $T$ starting at the unused boundary component (see Figure~\ref{fig-slit}), and cutting along these slits, all leaves of the obtained foliation are now compact and the dual tree provides a simplicial approximation $S_t$ of $T$, which naturally comes with a morphism $f_t:S_t\to T$. We observe that all morphisms $f_t$ with $t>0$ sufficiently small have the same turning class $\calb_{slit}(T)$: indeed, any turn defined by two directions not based at the vertex corresponding to the unused boundary component is in $\calb_{slit}(T)$, and a turn based at the vertex corresponding to the boundary component is in $\calb_{slit}(T)$ if and only if it is made of two directions that are separated by only one singular leaf of the foliation. Using Lemma~\ref{ubi-0}, we see that $\calb_{slit}(T)$ is nothing but the collection of all ubiquitous turns in $T$.  
The map $\Name$ sending the turning class $\calb_{slit}(T)$ to  $0$, and any other turning class to $\perp$, satisfies the conditions from Definition~\ref{de-namable} (the fact that it is measurable follows from Lemma~\ref{ubi-mes}).
\end{proof}

\subsubsection{The turning class is of finite type.}

We are now left understanding relatively free geometric arational trees. 
We will now prove that turning classes of morphisms with such a tree as their target 
satisfy some finiteness properties. We make the following definition.

\begin{de}[The set $\calb_{[d],[d']}$]\label{sett}
Let $T\in \baro$, and let $\calb$ be a turning class in $T$.
Let $x\in T$ be a branch point, and let $[d],[d']$ be two $G_x$-orbits of directions based at $x$. 
\\ We denote by $\calb_{[d],[d']}\subseteq \calb$ the set of turns $(\delta,\delta')\in \calb$ with $\delta\in[d]$ and $\delta'\in [d']$.
\\ We say that $\calb_{[d],[d']}$ is \emph{full} if it is equal to $[d]\times [d']$ minus the diagonal.
\end{de}

Notice that the set $\calb_{[d],[d']}$ is $G_x$-invariant.

\begin{de}[Turning class of finite type]
A turning class $\calb$ of a tree $T\in\baro$ is \emph{of finite type} if for each branch point $x\in T$ and for each pair of $G_x$-orbits of directions $[d],[d']$ at $x$, the set $\calb_{[d],[d']}$ is either full or finite modulo $G_x$.
\end{de}

Notice that the set of all turning classes of finite type on trees in $\overline{\calo}$ forms a measurable subset of the set $\oturn$ of all turning classes
(see Section \ref{sec_algebra}).

\begin{lemma}\label{lem_finite_type}
Let $T\in \baro$ be geometric, with trivial arc stabilizers.
Assume that $T$ is relatively free. 
\\ Then for any $S\in \Os$ and any morphism $f:S\ra T$, the turning class of $f$ is of finite type.
\end{lemma}

\begin{proof}
Let $d,d'$ be a turn based at a branch point $x\in T$. We aim to show that $\calb_{[d],[d']}$ is either $G_x$-cofinite or full.

We first claim that all but finitely many $G_x$-orbits of turns in $\calb_{[d],[d']}$ lift at a vertex of $S$ with non-trivial stabilizer. To prove this claim, we first note that if two (ordered) turns $(\delta,\delta'_1)$, $(\delta,\delta'_2)$ at $x$ are in the same $G_x$-orbit, then they have to coincide since the stabilizer of any direction is trivial.
It follows that $\calb_{[d],[d']}$ contains at most finitely many $G_x$-orbits of ubiquitous turns since there are only finitely many orbits of ubiquitous turns at branch points in $T$. 
Now for each edge $e$ of $S$, the segment $f(e)$ crosses only finitely many non-ubiquitous turns. It follows that turns taken by the images of the edges of $S$
contribute only finitely many orbits of turns in $\calb_{[d],[d']}$. Since $S$ has only finitely many $G$-orbits of turns at vertices with trivial stabilizer, the claim follows. 

If no turn in $\calb_{[d],[d']}$ lifts at a vertex of $S$ with nontrivial stabilizer, then our claim shows that $\calb_{[d],[d']}$ is $G_x$-cofinite.
So consider a vertex $s\in f\m(x)$ with nontrivial stabilizer, and a turn $(\Tilde \delta_1,\Tilde \delta_2)$ at $s$ that maps to some $(\delta_1,\delta_2)\in \calb_{[d],[d']}$.
Since $T$ is relatively free, 
$s$ is the unique vertex with nontrivial stabilizer in $f\m(x)$, and $G_s=G_x$.
Then for all $g_1,g_2\in G_s=G_x$, the turn $(g_1\Tilde\delta_1,g_2\Tilde\delta_2)$ is mapped to $(g_1\delta_1,g_2\delta_2)$ so $\calb_{[d],[d']}$ is full.
\end{proof}

\subsubsection{Definition of $\Name$ in the indecomposable case} \label{sec_indecomposable}

We denote by $\FGeom$ the collection of all geometric, relatively free arational trees. This is a Borel subset of $\AT$, being the complement of the union $\mathcal{NG}eom \cup\SAT$. Every tree $T\in\FGeom$ has a natural decomposition given by a transverse covering by indecomposable trees \cite{Gui}.
Since $T$ is arational, it has to be mixing (see \cite[Proposition~8.3]{Rey} for free groups and \cite[Lemma~4.9]{Hor2} in general), i.e.\ 
there is exactly one orbit of indecomposable trees in the transverse covering. 
The skeleton of this transverse covering is a bipartite simplicial tree which has one orbit of vertices corresponding to the indecomposable pieces, whose stabilizer is finitely generated, while the stabilizer of every other vertex in $S$ is a peripheral group (because $T$ is relatively free). 
Edge stabilizers of this skeleton are finitely generated.
In particular, the set $\cals k$ of all possible skeleta of trees in $\FGeom$ is countable. We denote by $\mathcal{GI}$ the subset of $\FGeom$ made of all indecomposable geometric relatively free arational trees in $\baro$.

\begin{lemma}\label{comp-mes}
Let $S$ be a simplicial $G$-tree. Then the set of all trees in $\baro$ that are compatible with $S$ is a Borel subset of $\baro$.
\end{lemma}

\begin{proof}
This is because a tree $T$ is compatible with $S$ if and only if the sum of the translation length functions of $T$ and $S$ is a translation length function of some $G$-tree \cite[Theorem~A.10]{GL2}, and the space of translation length functions of $G$-trees is closed (hence a Borel subset) in $\mathbb{R}^G$, see \cite{CM}.
\end{proof}

\begin{cor}\label{gi-mes}
The set $\mathcal{GI}$ is a Borel subset of $\mathcal{G}eom$.
\end{cor}

\begin{proof}
This is a consequence of Lemma~\ref{comp-mes} because a geometric tree is indecomposable if and only if it is not compatible with any simplicial tree in $\cals k$, and $\cals k$ is countable. In addition, as already observed in the previous section, being relatively free is a Borel condition. 
\end{proof}

The map $\Name$ will be defined in terms of the following notion of angle.

\begin{de}\label{def_angle}
Let $T\in\GI$, and $d,d'$ be two directions in $T$ based at the same branch point $x\in T$.
\\ The \emph{angle} $\angle (d,d')$ between $d$ and $d'$
is the minimal number $k$ of ubiquitous turns $U_1,\dots,U_k$ such that $d$ is a direction in $U_1$, $d'$ is a direction in $U_k$, and for all $i\in\{1,\dots,k-1\}$, the turns $U_i$ and $U_{i+1}$ have a direction in common.  
\end{de}

The angle can be interpreted as a distance in the Whithead graph at $x$ defined by ubiquitous turns. It is finite because this Whitehead graph is connected (Lemma~\ref{lem_whitehead}). Obviously, the angle is $G$-invariant. Moreover, since by Corollary \ref{ubi} any direction is part
of only finitely many ubiquitous turns, the Whitehead graph is locally finite. In other words,  
for any direction $d$, and $k\in \bbN$,
there are only finitely many directions $d'$ based at the same point as $d$ such that $\angle(d,d')\leq k$. 

We now define the name of a turning class $\calb$ on $T$ as follows.
If $\calb$ is not of finite type, we define its name as $\perp$. 
We now assume that $\calb$ is of finite type.
Given a branch point $x\in T$ and
a pair $([d],[d'])$ of $G_x$-orbits of directions based at $x$, we let $N_{[d],[d']}=0$ if $\calb_{[d],[d']}$ is full, and otherwise we let 
$$N_{[d],[d']}=\max_{(\delta,\delta')\in \calb_{[d],[d']}}\angle(\delta,\delta').$$
This is well defined because $\calb$ is of finite type.
We then let $$\Name(\calb)=\max_{([d],[d'])}(N_{[d],[d']}),$$ where the maximum is taken over all pairs of orbits of directions at branch points.

\begin{prop}\label{name-indec}
The set $\GI$ has namable turning classes: 
the map $$\Name:\turn\GI\ra \mathbb{N}\cup\{\perp\}$$ satisfies all properties from Definition~\ref{de-namable}. 
\end{prop}

\begin{proof}
The fact that $\Name$ is  measurable can easily be checked using Lemmas~\ref{dir-mes} and~\ref{ubi-mes} (representatives of the pairs of orbits of directions in $T$ can be chosen in a measurable way). The second and third   conditions from Definition~\ref{de-namable} are clear from the definition of $\Name$. 
The fourth   condition follows from the fact that there are only finitely many orbits of pairs of directions making a given angle,
and there are finitely many possibilites for the pairs of orbits of directions for which $\calb_{[d],[d']}$  is full.
The fifth   condition follows from Lemma~\ref{lem_finite_type}.
\end{proof}

\subsubsection{Definition of $\Name$ in the decomposable case} 

The goal of the present section is to prove that the set of all geometric trees in $\AT$ has namable turning classes, without restricting to the indecomposable ones.

We start by defining a map $\Name$ on the set $\turn\FGeom$ of turning classes on relatively free, arational, geometric trees.
Let $T\in\FGeom$, and let $\calb$ be a turning class on $T$. 
Let $\caly$ be the transverse covering of $T$ by its maximal indecomposable subtrees (recall that all trees in $\caly$ are in the same orbit because $T$ is mixing). We say that a turn $(d,d')$ in $T$ \emph{crosses} $\caly$ if it is not contained in any $Y\in \caly$.
We say that the turning class $\calb$ is \emph{acceptable} if given any turn $(d,d')\in\calb$ based at $x\in T$ and $g\in G_x$ such that both $(d,d')$ and $(d,gd')$ cross $\caly$, then $(d,d')\in\calb$ if and only if $(d,gd')\in\calb$. 
We then say that $\calb$ \emph{crosses $\caly$ along $([d],[d'])$} when there exists a turn $(\Tilde d,\Tilde d')\in \calb$ with $\Tilde d\in [d]$, and 
$\Tilde d'\in [d']$ such that $(\Tilde d,\Tilde d')$ crosses $\caly$ (in which case, all turns in $[d]\times [d']$ crossing $\caly$ are in $\calb$ if $\calb$ is acceptable).

If $\calb$ is not acceptable or is not of finite type, we define its name as $\perp$. 
Otherwise, choose $Y\in \caly$ and let $\calb_{|Y}$ be the turning class of $Y$ obtained by restriction of $\calb$.
We then define $\Name(\calb)=\Name(\calb_{|Y})$. Since $\Name(\calb_{|Y})$ is defined as an angle when not full
this name does not depend on the choice of $Y\in \caly$.

\begin{prop}\label{name-dec} 
The set $\FGeom$ has namable turning classes: 
the map $$\Name:\turn\FGeom\ra \mathbb{N}\cup\{\perp\}$$ satisfies all properties from Definition~\ref{de-namable}. 
\end{prop}

\begin{proof}
The map $\Name$ clearly satisfies the second and third conditions.

Let us check that $\Name$ satisfies Assertion 4. 
Given $n\neq \perp$, and a turning class $\calb$  on $T$ with name $n$, $\calb$ is acceptable, so $\calb$ is completely determined by $\calb_{|Y}$ and by the (finite) collection of all pairs $([d],[d'])$ of orbits of directions along which $\calb$ crosses $\caly$ (with the above notations). Finiteness thus follows from finiteness in the indecomposable case (Proposition~\ref{name-indec}).

We now prove Assertion 5. 
Given $T\in \FGeom$, and $Y\in\caly$ as above,
let $G_Y\actson R$ be any Grushko tree for $(G_Y,\calf_{|G_Y})$ with a very optimal morphism $f_Y:R\ra Y$.
Let $S$ be the skeleton of $\caly$. 
It has a vertex $v_Y$ whose stabilizer is $G_Y$.
Blow-up in an equivariant way the vertex $v_Y$ of $S$ into $R$ and attach each edge $e$ incident on $v_Y$ to the unique point fixed by $G_e$ (recall that $G_e$ is non-trivial because $T$ is arational, and peripheral because $T$ is relatively free). 
Let $\hat S$ be the obtained blown-up tree, and let $S_0$ be the tree obtained from $\hat S$ by collapsing all the edges coming from $S$. 
The tree $S_0$ lies in $\Os$, 
and we denote by $\calr$ the transverse covering of $S_0$ by the translates of $R$.
The morphism $f_Y:R\ra Y$ extends uniquely to a map $\hat S\ra T$ which is constant on all edges coming from $S$.
This yields a morphism $f:S_0\ra T$ which is   very optimal because $f_Y$ is. 
The turning class $\calb$ of $f$ is of finite type by Lemma \ref{lem_finite_type}. 

To ensure that $\Name(\calb)\neq\perp$, we now check that $\calb$ is acceptable. 
Let $(d_1,d_2)$ be a turn in $\calb$ that crosses $\caly$. Let $x$ be the base point of this turn, and $Y_1,Y_2\in \caly$ the subtrees containing $d_1$ and $d_2$ respectively. Let $R_1,R_2\subset S_0$ be the preimages of $Y_1,Y_2$ in $S_0$ (these are translates of $R\subset S_0$). 
These two subtrees intersect in a single vertex $v$ coming from the copy of the point $x$ in the skeleton, so $G_x=G_v$ and $f(v)=x$.
Let $(\Tilde d_1,\Tilde d_2)$ be a lift of $(d_1,d_2)$ in $S_0$. Since $\tilde d_i$ is contained in $R_i$, $\tilde d_1$ and $\tilde d_2$ are necessarily based at $v$.
Then for all $g\in G_x=G_v$, the turn $(\Tilde d,g\Tilde d')$ is a lift of $(d,gd')$. In particular, all turns $(d,gd')$ (and in particular all which cross $ \caly$) are in $\calb$.

We finally establish that $\Name$ is measurable. 
We claim that crossing is a Borel condition, i.e.\ for all $n,m\in\mathbb{N}$, the set  
$$\{T\in\FGeom|(d_n(T),d_m(T))\text{~defines a turn that crosses~}\caly\}$$ is a Borel subset of $\FGeom$.
Indeed, if $d_n(T),d_m(T)$ are based at the same point, then $(d_n(T),d_m(T))$ does not cross
if and only if for every arc $I\subset T$, there exists $g_1,\dots,g_N\in G$ such that $d_n(T)$ is contained in $g_1.I$,
$d_m(T)$ is contained in $g_N.I$, and $g_i I\cap g_{i+1} I$ is non-degenerate for all $i<N$.
This condition does not change if we impose that the endpoints of $I$ are branch points of $T$.
The above characterization can therefore be expressed using the measurable enumeration of branch points, thus proving the claim.
Since the name of a turning class is expressed in terms of angles of such turns, the measurability of $\Name$ easily follows.
\end{proof}

\section{Proof of the main theorem and applications}\label{sec-appl}

\subsection{Proof of the main theorem} 

We start by establishing the following result.

\begin{theo}\label{arat-2}
Let $G$ be a countable group, and let $\calf$ be a free factor system of $G$.
\\ Then there exists a sequence of Borel maps $$\mu_n:\mathbb{P}\AT\to\mathrm{Prob}(\Simp)$$ such that for all $\Phi\in\Out(G,\calf)$ and all $T\in\mathbb{P}\AT$, one has $$||\Phi.\mu_n(T)-\mu_n(\Phi.T)||_1\to 0$$ as $n$ goes to $+\infty$.
\end{theo}

\begin{proof}
By Proposition \ref{prop_namable}, $\AT$ has namable turning classes.
Theorem~\ref{arat-2} thus follows from Proposition~\ref{name-enough}.
\end{proof}

We are now in position to complete the proof of our main theorem.

\begin{theo}\label{main-2}
Let $\{G_1,\dots,G_k\}$ be a finite   family of countable groups, and let $$G:=G_1\ast\dots\ast G_k\ast F_N,$$ where $F_N$ is a free group of rank $N$.
\\ Assume that for all $i\in\{1,\dots,k\}$, the group $G_i$ is   boundary amenable. 
\\ Then $\mathrm{Out}(G,\{G_i\}^{(\mathrm{t})})$ is   boundary amenable.
\end{theo}

\begin{proof}[Proof of Theorem~\ref{main-2}]
As usual, we denote by $\calf$ the finite   family of the $G$-conjugacy classes of the subgroups $G_i$. The proof is by induction on the complexity $\xi(G,\calf)=(N+k-1,N)$, where complexities are ordered with respect to the lexicographic order. We start with the sporadic cases.
\begin{itemize}
\item The result is obvious if either $\calf=\{[G]\}$ (i.e.\ $\xi(G,\calf)=(0,0)$) or $G=\mathbb{Z}$ (i.e.\ $\xi(G,\calf)=(0,1)$).
\item  If $G=G_1\ast G_2$ and $\calf=\{[G_1],[G_2]\}$ (i.e.\ $\xi(G,\calf)=(1,0)$), then the group $\Out(G,\calf^{(\mathrm{t})})$ is isomorphic to $G_1/Z(G_1)\times G_2/Z(G_2)$ (see \cite{Lev}), which is   boundary amenable, being a direct product of two groups that are   boundary amenable by Proposition \ref{prop_quotient_amenable}.  
\item  If $G=G_1\ast\mathbb{Z}$ and $\calf=\{[G_1]\}$ (i.e.\ $\xi(G,\calf)=(1,1)$), then  $\Out(G,\calf^{(\mathrm{t})})$ has an index $2$ subgroup which is isomorphic to $(G_1\times G_1)/Z(G_1)$, where $Z(G_1)$ embeds diagonally in $G_1\times G_1$ (see \cite{Lev}). The group $(G_1\times G_1)/Z(G_1)$ maps to $G_1/Z(G_1)\times G_1/Z(G_1)$ with central kernel, and is therefore   boundary amenable by Propositions~\ref{ext-exact} and  \ref{prop_quotient_amenable}. 
\end{itemize}
 We now assume that $(G,\calf)$ is non-sporadic. It was proved in \cite{Rey,Hor2} that to any (projective class of) non-arational tree, one can associate a canonical finite set of conjugacy classes of free factors: this means that there exists an $\text{Out}(G,\calf)$-equivariant map from $\mathbb{P}\overline{\calo}\setminus\mathbb{P}\AT$ to the countable collection   $\mathcal{F}(\mathrm{FF})$ of finite sets of   onjugacy classes of proper $(G,\calf)$-free factors. Notice also that there exists an $\Out(G,\calf)$-equivariant map $\pi:\Simp\to\calf(\mathrm{FF})$, sending a simplex $\Delta$ to the collection of all   conjugacy classes of proper free factors that are elliptic in some tree obtained by collapsing some edges in the underlying tree of $\Delta$. Combining these facts with Theorem~\ref{arat-2}, we get a sequence of Borel maps 
$$\mu_n:\mathbb{P}\baro\to\mathrm{Prob}(\mathrm{FF})$$ 
such that for all $\Phi\in\Out(G,\calf)$ and all $T\in\mathbb{P}\baro$, one has 
$$||\Phi.\mu_n(T)-\mu_n(\Phi.T)||_1\to 0$$ 
as $n$ goes to $+\infty$. This holds in particular for every $\Phi\in \Out(G,\calf^{(\mathrm{t})})$.
Since $\mathbb{P}\overline{\calo}$ is compact and $\mathrm{FF}$ is countable, using  Corollary~\ref{ozawa-kida}, it is enough to prove that the stabilizer in $\Out(G,\calf^{(\mathrm{t})})$ of   the conjugacy class of any proper $(G,\calf)$-free factor is   boundary amenable. Let $A$ be a proper $(G,\calf)$-free factor, and let $\calf'$ be the smallest $(G,\calf)$-free factor system that contains $A$. Then there is a morphism from the stabilizer of $A$ in $\Out(G,\calf^{(\mathrm{t})})$ to $\Out(A,\calf_{|A}^{(\mathrm{t})})$, whose kernel is contained in $\Out(G,\calf'^{(\mathrm{t})})$. An easy computation shows that $\xi(A,\calf_{|A})<\xi(G,\calf)$ and $\xi(G,\calf')<\xi(G,\calf)$, so arguing by induction on the complexity (and using the fact that extensions of   boundary amenable countable groups are   boundary amenable), we get that the stabilizer of $A$ in $\Out(G,\calf^{(\mathrm{t})})$ is   boundary amenable.
\end{proof}

\begin{cor}\label{cor-main}
Let $\{G_1,\dots,G_k\}$ be a finite   family of countable groups, and let $$G:=G_1\ast\dots\ast G_k\ast F_N,$$ where $F_N$ is a free group of rank $N$.
\\ Assume that for all $i\in\{1,\dots,k\}$, the groups $G_i$ and $\Out(G_i)$ are   boundary amenable. 
\\ Then $\text{Out}(G,\{G_i\})$ is   boundary amenable.  
\end{cor}

\begin{proof}
This follows from the fact that there is a short exact sequence $$1\to\Out(G,\calf^{(\mathrm{t})})\to\Out(G,\calf)\to\prod_{i=1}^k\Out(G_i)\to 1,$$ and   boundary amenability is stable under extensions (Proposition~\ref{ext-exact}). 
\end{proof}

\subsection{Automorphisms of relatively hyperbolic groups}\label{sec_RH}

\begin{cor}\label{cor-hyp}
Let $G$ be a torsion-free group which is hyperbolic relative to a finite   family of finitely generated groups $\mathcal{P}:=\{P_1,\dots,P_k\}$. 
\\ Assume that for all $i\in\{1,\dots,k\}$, the group $P_i$ is   boundary amenable. 
\\ Then $\Out(G,\calp^{(t)})$ is   boundary amenable.
\end{cor}

Before proving Corollary~\ref{cor-hyp}, we establish the following consequence.

\begin{cor}\label{cor-toral}
Let $G$ be a torsion-free group which is hyperbolic relative to a finite   family of finitely generated groups $\mathcal{P}:=\{P_1,\dots,P_k\}$. 
\\ Assume that for all $i\in\{1,\dots,k\}$, the groups $P_i$ and $\Out(P_i)$ are   boundary amenable.
\\ Then $\Out(G,\mathcal{P})$ is   boundary amenable.
\\ In particular, the outer automorphism group of any torsion-free Gromov hyperbolic group (or more generally any toral relatively hyperbolic group) is   boundary amenable.
\end{cor}

\begin{proof}
The first statement follows from Corollary~\ref{cor-hyp} as in Corollary~\ref{cor-main}. The application to torsion-free hyperbolic groups is immediate. The case of toral relatively hyperbolic groups follows from the fact that if some parabolic subgroup is isomorphic to $\mathbb{Z}^n$, then its outer automorphism group is linear, whence   boundary amenable by \cite{GHW}; in addition, every automorphism of a toral relatively hyperbolic group permutes the conjugacy classes of all non-cyclic parabolic subgroups.
\end{proof}

\begin{proof}[Proof of Corollary~\ref{cor-hyp}]
First assume that $G$ is freely indecomposable relative to the parabolic subgroups. 
This case follows easily from   boundary amenability of mapping class groups \cite{Kid,Ham} and JSJ theory \cite{Sela_JSJ,Lev,GL4} as we now explain.
By \cite{GL4}, the group $\text{Out}(G,\mathcal{P}^{(\mathrm{t})})$ has a finite index subgroup that maps onto a product of mapping class groups of   compact surfaces, with kernel contained in the group $\mathcal{T}$ of twists of the canonical elementary JSJ decomposition of $G$. The group $\mathcal{T}$ maps with abelian kernel to a direct product of copies of $P_i/Z(P_i)$ \cite[Appendix, Lemma~5]{Hae}, so $\calt$ is   boundary amenable. 
By \cite{Kid,Ham}, the mapping class group of an orientable   compact surface $S$ is   boundary amenable. This also holds if $S$ is non-orientable because $\MCG(S)$ is commensurable
to a subgroup of $\MCG(\Tilde S)$ where $\Tilde S$ is the orientation cover of $S$.
Therefore, $\text{Out}(G,\mathcal{P}^{(\mathrm{t})})$ is   boundary amenable.

We now assume that $G$ splits as a free product of the form $$G=G_1\ast\dots\ast G_k\ast F_N$$ relative to the parabolic subgroups, with $G_i$ freely indecomposable relative to $\calp_{|G_i}$. Then the group $\Out(G,\calp^{(\mathrm{t})})$ has a finite index subgroup $\Out^0(G,\calp^{(\mathrm{t})})$ which does not permute the conjugacy classes of the factors $G_i$. There is a morphism $$\Out^0(G,\mathcal{P}^{(\mathrm{t})})\to\prod_{i=1}^k\Out(G_i,\mathcal{P}_{|G_i}^{(\mathrm{t})}),$$ whose kernel is equal to $\Out(G,\{G_i\}^{(\mathrm{t})})$. The above paragraph shows that each group $\Out(G_i,\mathcal{P}_{|G_i}^{(\mathrm{t})})$ is   boundary amenable, and therefore their direct product is   boundary amenable. It is thus enough to check that the kernel $\Out(G,\{G_i\}^{(\mathrm{t})})$ is   boundary amenable. Since each $G_i$ is hyperbolic relative to $\mathcal{P}_{|G_i}$ and relatively hyperbolic groups with   boundary amenable parabolic subgroups are   boundary amenable \cite{Oza}, we deduce that each $G_i$ is   boundary amenable. Therefore, Theorem~\ref{main-2} shows that $\Out(G,\{G_i\}^{(\mathrm{t})})$ is   boundary amenable, which concludes the proof.
\end{proof}

We conclude this subsection with the following extension.

\begin{cor}\label{cor_virtually_tf}
Let $G$ be a virtually torsion-free hyperbolic group. Then $\Out(G)$ and $\Aut(G)$ are
boundary amenable.
\end{cor}

\begin{proof}
    Let $G_0\subset G$ be a characteristic subgroup of finite index with $G_0$ torsion-free.
    Since $\Aut(G_0)$ embeds as a subgroup of $\Out(G_0*\bbZ)$, both $\Aut(G_0)$ and $\Out(G_0)$ are boundary amenable by Corollary \ref{cor-toral}.
    
    We are going to use several times that if $H$ maps to $H'$ with finite kernel, then boundary amenability of $H'$ implies that of $H$ (Corollary \ref{cor-subgroup} and Proposition \ref{ext-exact}), and the converse holds if the map is onto (Proposition \ref{prop_quotient_amenable}).
    
    The restriction map $\Aut(G)\ra \Aut(G_0)$ has finite kernel so boundary amenability
    of $\Aut(G)$ follows from boundary amenability of $\Aut(G_0)$.
    
    This map induces a homomorphism $\Aut(G)/\Inn(G_0)\ra \Aut(G_0)/\Inn(G_0)=\Out(G_0)$
    whose kernel is finite, showing that $\Aut(G)/\Inn(G_0)$ is boundary amenable.
    Since the epimorphism $\Aut(G)/\Inn(G_0)\ra \Aut(G)/\Inn(G)=\Out(G)$ has finite kernel
    it follows that $\Out(G)$ is boundary amenable.
\end{proof}

\subsection{Automorphisms of right-angled Artin groups} \label{sec_RAAG}

We recall that given a   simplicial graph $X$, the \emph{right-angled Artin group} $A_X$ is the group whose generators are the vertices of $X$, in which the relations are given by commutation of any two generators corresponding to vertices of $X$ that are joined by an edge.   A right-angled Artin group $A_X$ is finitely generated if and only if $X$ is a finite graph. Using Theorem~\ref{main-2} and previous work of Charney--Vogtmann~\cite{CV}, we will derive the following statement.

\begin{cor}\label{raag}
For any   finitely generated right-angled Artin group $A$, the group $\Out(A)$ is   boundary amenable.
\end{cor}

\begin{proof}
The proof is by induction on the number of vertices of the defining graph $X$ of $A$. 

If $A$ has nontrivial center, then by \cite[Proposition 4.4]{CV}, there exists an equivalence class $[v]$ of vertices in $X$ (with the terminology from \cite{CV}), and a finite index subgroup $\text{Out}^0(A)$ of $\text{Out}(A)$, such that there is a morphism $$\text{Out}^0(A)\to\text{Out}(A_{[v]})\times\text{Out}(A_{\text{lk}([v])})$$ with abelian kernel. If $A$ is abelian, then $\text{Out}(A)$ is linear, whence   boundary amenable \cite{GHW}. Otherwise both $[v]$ and $\text{lk}([v])$ are proper subgraphs of $X$, so we can apply our induction hypothesis to deduce that $\text{Out}(A)$ is   boundary amenable.

If $X$ is connected, and the center of $A$ is trivial, then by \cite[Corollary 3.3]{CV}, the group $\text{Out}(A)$ has a finite index subgroup $\text{Out}^0(A)$ that admits a morphism $$\text{Out}^0(A)\to\prod\text{Out}(A_{\text{lk}([v])})$$ whose kernel is abelian \cite[Theorem~4.2]{CV}, where the product is taken over all maximal abelian equivalence classes of vertices in $X$. Again the result follows from our induction hypothesis.

We finally assume that $X$ is disconnected. Then $A$ splits as a free product $$A=A_{X_1}\ast\dots\ast A_{X_k}\ast F_N,$$ where $X_1,\dots,X_k$ are the connected components of $X$ that are not reduced to a point, and $X$ has $N$ connected components reduced to a point. 
The group $A_{X_i}$ is   boundary amenable \cite{CN}, and $\Out(A_{X_i})$ is   boundary amenable by induction hypothesis.
Corollary \ref{cor-main} implies that $\Out(A,\{A_{X_i}\})$ is   boundary amenable. Since it has finite index in $\Out(A)$, the group $\Out(A)$ is   boundary amenable.   
\end{proof}

\section{Complements}\label{sec-complements}

\subsection{An explicit compact space with an amenable action of $\Out(F_N)$}\label{sec-compact}

We will now give an explicit compact 
space with a topologically amenable $\Out(F_N)$-action. 
This could also be generalized to any free product given compact spaces on which the groups $G_i/Z(G_i)$ and $\Out(G_i)$ act amenably. 

\newcommand{\SF}{\mathrm{SF}}
Define a \emph{subfactor system} in $F_N$ as the conjugacy class of a pair $(A,\calf)$ where $A\subseteq F_N$ is a non-trivial free factor (possibly $A=F_N$), and $\calf$ is a proper free factor system of $A$ (possibly empty, but $\calf\neq \{A\}$).
We denote by $\SF$ the set of all subfactor systems of $F_N$.

To each subfactor system $(A,\calf)$ we associate a compact space $\Omega_{(A,\calf)}$.
These spaces will depend naturally on $(A,\calf)$ in the following sense:
    for each $\Phi\in \Out(F_N)$, there will be a natural homeomorphism $h_\Phi:\Omega_{(A,\calf)}\to \Omega_{\Phi(A,\calf)}$ 
    such that $h_{\Phi\circ\Psi}=h_\Phi\circ h_\Psi$ for all $\Phi,\Psi\in\Out(F_N)$. The space we consider is then the product space $$\Omega:=\prod_{(A,\calf)\in\SF}\Omega_{(A,\calf)},$$ 
on which the $\Out(F_N)$-action is defined by 
$$\Phi.(\omega_{(A,\calf)})_{(A,\calf)\in \SF}=(h_\Phi(\omega_{\Phi\m(A,\calf)}))_{(A,\calf)\in \SF}.$$
   
The spaces $\Omega_{(A,\calf)}$ are defined as follows. For non-sporadic subfactor systems $(A,\calf)$, we define $\Omega_{(A,\calf)}$ as the projectivization $\mathbb{P}\baro(A,\calf)$ of the closure of the associated Outer space. When $A=\mathbb{Z}$, we let $\Omega_{(A,\calf)}$ be a point. If  $A=A_1\ast A_2$ and $\calf=\{[A_1],[A_2]\}$, then we let $\Omega_{(A,\calf)}:=\partial_\infty A_1\times\partial_\infty A_2$. Finally, if
 $A=A_1\ast\mathbb{Z}$ and $\calf=\{[A_1]\}$, we let $\Omega_{(A,\calf)}:=(\partial_\infty A_1)^2$.

The naturality is obvious in the first two cases, let us make it explicit in the last two cases. We will only construct $h_\Phi$ when $\Phi$ belongs to the stabilizer of $(A,\calf)$, since this allows to compute $h_\Phi$ in general.

If $A=A_1\ast A_2$ and $\calf=\{[A_1],[A_2]\}$ then $\Out(A,\calf)$ is isomorphic to $\Aut(A_1)\times\Aut(A_2)$: this is shown by observing that every element $\Phi\in\Out(A,\calf)$ has a unique representative $\phi\in\Aut(A)$ that fixes both subgroups $A_1$ and $A_2$ (as opposed to just fixing the conjugacy classes of these two subgroups); the isomorphism from $\Out(A,\calf)$ to $\Aut(A_1)\times\Aut(A_2)$ then maps $\Phi$ to the pair $(\phi_1,\phi_2)$ made of the restrictions of $\phi$ to $A_1$ and $A_2$. Therefore, the group $\Out(A,\calf)$ acts by homeomorphisms on $\partial_\infty A_1\times \partial_\infty A_2=\Omega(A,\calf)$.  

If    $A=A_1\ast\mathbb{Z}$ with $\calf=\{[A_1]\}$, let   $T$ be the tree whose quotient graph has one vertex with vertex group $A_1$ and a single loop-edge with trivial edge group. 
We choose a base edge $e=uv$ where the stabilizer of $u$ is $A_1$, and we let $t$ be an element of $A$ sending $u$ to $v$ (thus $t$ is a stable letter of the HNN extension). The group $\Aut(A,\calf)$ of (not outer) automorphisms of $A$ preserving the conjugacy class of $A_1$ acts on $T$.
Any outer automorphism $\Phi\in\Out(A,\calf)$ has a unique representative $\phi\in \Aut(A,\calf)$ that preserves the edge $e$ (it might exchange its endpoints). Looking at the restriction of $\phi$ to $G_u=A_1$ and $G_v=tA_1t\m$, we get a map $\Out(A,\calf)\ra (\Aut(A_1)\times \Aut(A_1))\rtimes \bbZ/2\bbZ$.
This group acts by homeomorphisms on $\partial_\infty A_1\times\partial_\infty A_1=\Omega(A,\calf)$. 

\begin{theo}
The $\text{Out}(F_N)$-action on $\Omega$ is topologically amenable.
\end{theo}

The proof is similar to the proof of Theorem \ref{main-2}, keeping track of the spaces when using Ozawa's inductive procedure (Proposition \ref{ozawa} and Corollary \ref{cor-xy}).
 The base case of the induction follows from the fact that when $(A,\calf)$ is sporadic,
the $\Out(A,\calf^{(\mathrm{t})})$-action on $\Omega_{(A,\calf)}$ is topologically amenable, because the natural action of a hyperbolic group on its boundary is topologically amenable. We leave the details to the reader.

\begin{rk}\label{poor-man}
If one takes for $\Omega'$ the product of Outer spaces of all non-elementary subgroups of $F_N$, the natural action
on this space (and even on the closure of the diagonal embedding of the projectivized Outer space) is not topologically amenable for $N\geq 3$.
Indeed, consider a decomposition $F_N=A*B$ with $A$ not cyclic. 
Let $T$ be the Bass--Serre tree of this action, and $T_A$, $T_B$ be two free simplicial actions
of $A$ and $B$ respectively. This defines a point in $\Omega'$ as follows: let $H<F_N$ be a non-elementary subgroup; if the action of $H$ on $T$
is non-trivial, we take the minimal $H$-invariant subtree for this action as the $H$-coordinate. If not, then up to conjugating, we can assume that $H$ is contained in $A$ or $B$, 
and we take as $H$-coordinate the action of $H$ on the corresponding
tree $T_A$ or $T_B$. We now note that the stabilizer of this point of $\Omega'$ contains a subgroup isomorphic to $A$, hence is non-amenable.
Indeed, for all $a\in A$, the automorphism $\phi$ of $F_N$ that restricts to $\ad_a$ on $A$ and to the identity on $B$
fixes this point. 
\end{rk}

\begin{question}
In \cite{Ham}, Hamenstädt describes a compact space equipped with a topologically amenable action of the mapping class group of a surface, in terms of  complete geodesic laminations on the surface. A possible analogue for $\Out(F_N)$ might be to consider free actions on general $\Lambda$-trees, instead of just actions on $\mathbb{R}$-trees; the space $\Omega'$ in the above remark can be viewed as a baby model for this space of $\Lambda$-trees. Is the $\Out(F_N)$-action on this space topologically amenable? 
\end{question}

\subsection{On the amenability of the $\Out(G,\calf^{(\mathrm{t})})$-action on $\AT$}\label{sec-action-amen}

In Theorem \ref{thm:FN-case},
we proved that the action of $\Out(F_N)$ on the set of free arational trees is Borel amenable (this is the case where $\calf=\es$).
In fact Theorem~\ref{arat-2} implies that the action of $\Out(F_N)$ on the set of all arational trees in $\bbP\baro(F_N,\es)$ is Borel amenable because
the stabilizer of every simplex in $\bbP\calo(F_N,\es)$ is finite.

We now explain how to refine the proof of Theorem~\ref{arat-2} in order to show that the action of $\Out(G,\calf^{(\mathrm{t})})$ on $\AT$ is amenable 
even though the stabilizer of a simplex may be non-amenable.
It is important here to restrict to the subgroup $\Out(G,\calf^{(\mathrm{t})})$ (whereas Theorem~\ref{arat-2}  applies to $\Out(G,\calf)$): indeed, there are arational
trees whose stabilizer in $\Out(G,\calf)$ is non-amenable (one easily constructs examples where $T$ is an arational surface tree).
For similar reasons, this will also require further assumptions on the peripheral groups $G_i$ (see Remark \ref{rk_centralizer}). 
The idea of the proof is to replace simplices in $\calo$ by larger collections of simplices whose common stabilizer in $\Out(G,\calf^{(\mathrm{t})})$ is amenable.

\begin{theo}\label{at-amen} Let   $\calf=\{G_1,\dots,G_k\}$ be a finite   family of countable groups, and let
$G=G_1*\dots *G_k*F_N$ with $(G,\{G_i\})$ non-sporadic.
Assume that in each group $G_i$, centralizers of non-trivial elements are amenable.
\\ Then the   $\Out(G,\calf^{(\mathrm{t})})$-action on $\AT$ is Borel amenable.
\end{theo}

\begin{rk}\label{rk_centralizer}
The hypothesis on centralizers is necessary. 
Indeed, assume that one of the groups $G_i$ contains an element $a$ whose centralizer $Z_{G_i}(a)$ is non-amenable. Let $T$ be an arational surface $(G,\calf)$-tree where $G_i$ is amalgamated to one of the boundary curves (or conical points) along $a$. Then the stabilizer of $T$ contains a subgroup of twists isomorphic to $Z_{G_i}(a)/Z(G_i)$. It is therefore non-amenable, which prevents the $\Out(G,\calf^{(\mathrm{t})})$-action on $\AT$ from being Borel amenable, as noticed in Remark~\ref{rk-subgroup}. This construction requires to write $G$ as in Figure \ref{fig-arat-surf} in which the surface (or orbifold) holds an arational foliation.
This can be achieved by taking a sphere with $k\geq 4$ punctures or conical points, or a projective plane with 3 punctures or conical points \cite{AY}.
\end{rk}

\paragraph*{Stabilizers of pairs of Grushko trees.}  In the proof of Theorem~\ref{at-amen}, we will make use of the following proposition.

\begin{prop}\label{stab-pair}
Let $G$ be a countable group, and let $\calf$ be a free factor system of $G$, such that $(G,\calf)$ is non-sporadic. 
Assume that in each peripheral group $G_i$, centralizers of non-trivial elements are amenable.
\\ Let $S,S'\in\calo$ be two trees, such that there is no $(G,\calf)$-free splitting that is compatible with both $S$ and $S'$. 
\\ Then the common stabilizer of $S$ and $S'$ in $\Out(G,\calf^{(\mathrm{t})})$ is amenable.
\end{prop}

Notice that in the case where $G=F_N$ and $\calf=\emptyset$, the stabilizer of any point in Culler--Vogtmann's Outer space is finite, so the conclusion is obvious. From now on, we will assume that $\calf\neq\emptyset$. We denote by $\Aut(G,\calf^{(\mathrm{t})})$ the preimage in $\Aut(G)$ of $\Out(G,\calf^{(\mathrm{t})})$, by $\cala_S,\cala_{S'}\subseteq \Aut(G,\calf^{(\mathrm{t})})$ the preimage of the stabilizer of $S$ and $S'$ respectively, and we let $\cala_{S,S'}:=\cala_S\cap\cala_{S'}$. 
Given $g\in G$, we denote by $\ad_g$ the automorphism of $G$ given by the conjugation by $g$, i.e.\ $\ad_g(h)=ghg^{-1}$ for all $h\in G$.

\begin{lemma}
For all $\phi\in\cala_S$, there exists a unique isometry $H_\phi$ of $S$ which is $\phi$-equivariant, i.e.\ such that $H_\phi(gx)=\phi(g)H_\phi(x)$ for all $g\in G$.
\\ The map  
\begin{displaymath}
\begin{array}{cccc}
H:& \cala_{S} &\to &\mathrm{Isom}(S)\\
& \phi & \mapsto & H_\phi
\end{array}
\end{displaymath} 
\noindent is an injective group morphism, which sends $\ad_g$ to the isometry $x\mapsto gx$.
\end{lemma}

\begin{proof}
  Existence of $H_\phi$ follows from the definition of being in $\cala_S$, so we focus on uniqueness. Let $\phi\in\cala_S$, and let $H_\phi$ be a $\phi$-equivariant isometry of $S$. Let $v\in S$ be a vertex with nontrivial stabilizer. Then for all $g\in G_v$, we have $H_\phi(v)=H_\phi(gv)=\phi(g)H_\phi(v)$, so $H_\phi(v)$ is the only vertex of $S$ which is fixed by $\phi(g)$. Therefore the $H_\phi$-images of all vertices of $S$ with nontrivial stabilizer are completely determined. In addition, for every point $x\in S$, there exist three vertices $v_1,v_2,v_3$ of $S$ with nontrivial stabilizer so that $x$ belongs to the tripod spanned by $v_1,v_2,v_3$. Therefore, the $H_\phi$-image of every point in $S$ is determined. This shows that $H_\phi$ is unique. The fact that $H$ is a group morphism then follows from the observation that $$H_\phi\circ H_\psi(gv)=\phi\circ\psi(g)H_\phi\circ H_\psi(v)$$ for all $\phi,\psi\in\cala_S$, all $g\in G$ and all $v\in S$. Injectivity follows from the observation that if $H_\phi$ is the identity, then $\phi(g)x=gx$ for all $g\in G$ and all $x\in S$, so $\phi$ is the identity. The last statement of the lemma also follows because $x\mapsto gx$ is $\ad_g$-equivariant.
\end{proof}

For $\phi\in \cala_{S,S'}$, we will denote by $H_\phi$ (resp.\ $H'_\phi$) the isometry of $S$ (resp.\ $S'$) representing $\phi$. Then $\cala_{S,S'}$ acts on $S$ (resp.\ $S'$) by $\phi.x=H_\phi(x)$ (resp. $\phi.x=H'_\phi(x)$); these actions are denoted with a dot. We let $\cala\subseteq\cala_{S,S'}$ be the finite index subgroup made of all automorphisms acting as the identity on the quotient graphs $S/G$ and $S'/G$. Given a vertex $v$ of $S$ or $S'$, we denote by $G_v$ its stabilizer in $G$, and by $\cala_v$ its stabilizer for the $\cala$-action. 
Similarly,  we denote by $\cala_e$ the stabilizer of an edge $e$ for the $\cala$-action. 

\begin{lemma}\label{defo}
The actions $\cala\actson S$ and $\cala\actson S'$ belong to the same deformation space, i.e.\ there exist $\cala$-equivariant maps from $S$ to $S'$ and from $S'$ to $S$.
\end{lemma}

\begin{proof}
By symmetry of roles of $S$ and $S'$, it is enough to show that every point stabilizer for the action  $\cala\actson S$ fixes a point in $S'$. 
If $v\in S$ is a vertex with $G_v\neq\{1\}$, then $\cala_v=\{\phi\in\cala|\phi(G_v)=G_v\}$, and $\cala_v$ fixes the unique point $v'$ fixed by $G_v$ in $S'$.
Let now $w\in S$ be a vertex with $G_w=\{1\}$. Let $v\in S$ be a point with $G_v\neq\{1\}$ such that the segment $[w,v]$ does not meet any other point with nontrivial $G$-stabilizer. Since every automorphism in $\cala$ acts as the identity on the quotient graph $S/G$, for every $\phi\in\cala_w$, the isometry $H_\phi$ fixes $[v,w]$ pointwise. Therefore $\cala_w\subset \cala_v$ so $\cala_w$ fixes a point in $S'$.     
\end{proof}

\begin{lemma}\label{stabequal}
Let $e,e'\subseteq S$ be two edges. If $\cala_e\subseteq\cala_{e'}$, then $\cala_e=\cala_{e'}$.
\end{lemma}

\begin{proof}
Let $\phi\in\cala_{e'}$. Since $\phi$ acts trivially on $S/G$, there exists $g\in G$ such that $\phi.e=ge$, i.e.\ $\phi.e=\ad_g.e$. Then $(\ad_{g^{-1}}\circ\phi).e=e$, so by assumption $(\ad_{g^{-1}}\circ\phi).e'=e'$, so $ge'=\phi.e'=e'$. Since $S$ has trivial arc stabilizers for the $G$-action, it follows that $g=1$, so $\phi.e=e$, i.e.\ $\phi\in\cala_e$. 
\end{proof}

\begin{cor}\label{edge-stab-same}
The actions $\cala\actson S$ and $\cala\actson S'$ have the same edge stabilizers.
\end{cor}

\begin{proof}
By Lemma~\ref{defo}, there exists an $\cala$-equivariant map from $S$ to $S'$. Therefore, for every edge $e'\subseteq S'$, there exists an edge $e\subseteq S$ such that $\cala_e\subseteq\cala_{e'}$. By symmetry, there also exists an edge $e'_2\subseteq S'$ such that $\cala_{e'_2}\subseteq\cala_{e}$. We thus have $\cala_{e'_2}\subseteq\cala_e\subseteq\cala_{e'}$, and Lemma~\ref{stabequal} implies that these subgroups are all equal.
\end{proof}

Let $v\in S$ be a vertex, and $e\subseteq S$ be an edge incident on $v$. Then every automorphism in $\cala_e$ preserves $G_v$, and since $\cala\subset \Aut(G,\calf^{(\mathrm{t})})$, there is a restriction map $\rho_{v,e}:\cala_e\to\Inn(G_v)$. 
In the following statement, we denote by $\overline{\cala}$ the image of $\cala$ in $\Out(G,\calf^{(\mathrm{t})})$.

\begin{lemma}\label{nonamen}
Assume that $\overline{\cala}$ is non-amenable. Then there exist a vertex $v\in S$ and an edge $e$ incident on $v$ so that $\rho_{v,e}(\cala_e)$ is non-amenable.
\end{lemma}

\begin{proof}
Given a vertex $v\in S$, we let $E(v)$ be a set of representatives of the $G_v$-orbits of edges incident on $v$. We denote by $Z(G_v)$ the center of $G_v$, which we diagonally embed into $G_v^{E(v)}$. There is an injective morphism $$\theta:\overline{\cala}\to\prod_v \left(G_v^{E(v)}/Z(G_v)\right),$$ where the product is taken over a set of representatives of the $G$-orbits of vertices of $S$, defined as follows: if $\Phi\in\overline{\cala}$ and $\phi\in\cala$ is a representative fixing $v$ and acting as the identity on $G_v$, then to any incident edge $e$, one associates the element $g_{v,e}\in G_v$ such that $\phi.e=g_{v,e}e$; since $\phi$ is well defined modulo conjugation by an element of $Z(G_v)$, the tuple $(g_{v,e})_{e\in E(v)}$ is well defined modulo $Z(G_v)$. Since $G_v^{E(v)}/Z(G_v)$ maps to $(G_v/Z(G_v))^{E(v)}$ with central kernel, we get a map $$\theta':\overline{\cala}\to\prod_v\text{Inn}(G_v)^{E(v)}$$ with central kernel. The image of this map is contained in $\prod_v\prod_{e\in E(v)}\rho_{v,e}(\cala_e)$: indeed, if the representative $\phi$ of $\Phi$ acting as the identity on $G_v$ sends $e$ to $g_{v,e}e$, then $\Phi$ has a representative that fixes $e$ and acts by conjugation by $g_{v,e}$ on $G_v$. Since $\overline{\cala}$ is nonamenable, there exists a pair $(v,e)$ with $e\in E(v)$, such that $\rho_{v,e}(\cala_e)$ is nonamenable.
\end{proof}

\begin{proof}[Proof of Proposition~\ref{stab-pair}]
We assume that $\overline{\cala}$ is nonamenable, and we will prove that there exists a $(G,\calf)$-free splitting which is compatible with both $S$ and $S'$. 

Denote by $\cale$ the collection of edge stabilizers for the action $\cala\actson S$ (equivalently $\cala\actson S'$, in view of Corollary~\ref{edge-stab-same}). By Lemma~\ref{stabequal}, the trivial equivalence relation on $\cale$ is admissible in the sense of \cite[Definition~3.1]{GL5}. Cylinders of $S$ are defined by $$\text{Cyl}_e=\bigcup_{\cala_{e'}=\cala_e}e',$$ and the tree of cylinders $S^c$ is the bipartite simplicial tree having one vertex for each cylinder $C$, one vertex for each point $x\in S$ belonging to at least two cylinders, and an edge joining $x$ to $C$ whenever $x\in C$. 
Since $S$ and $S'$ belong to the same deformation space (Lemma~\ref{defo}), it follows from \cite[Theorem~1]{GL5} that $S^c=(S')^c$. From now on, we let $U:=S^c$. By \cite[Proposition~8.1]{GL5}, the tree $U$ is compatible with both $S$ and $S'$. Moreover, $U$ is a minimal $\cala$-tree by \cite[Lemma~4.9]{Gui04}, 
hence a minimal $G$-tree (the minimal $G$-tree is $\cala$-invariant because $\Inn(G)$ is normal in $\cala$).
Thus, it suffices to show that some edge of $U$ has trivial stabilizer in $G$.

Let $(v,e)$ be a pair as in Lemma~\ref{nonamen}, so that $\rho_{v,e}(\cala_e)$ is non-amenable. We claim that $\text{Cyl}_e\cap G_v.e=\{e\}$. Indeed, let $e'=ge$ for some $g\in G_v$, with $\cala_{e'}=\cala_e$. Then for all $\phi\in\cala_e$, one has $ge=e'=\phi.e'=\phi.(ge)=\phi(g)\phi.e=\phi(g)e$. Since edge stabilizers of the $G$-action on $S$ are trivial, we deduce that $\phi(g)=g$. Therefore $g$ is fixed by all the inner automorphisms in $\rho_{v,e}(\cala_e)$.
Since every non-trivial element of $G_v$ has an amenable centralizer, and $\rho_{v,e}(\cala_e)$ is non-amenable, $g=1$ and $e'=e$. This proves the claim. 

Let now $\varepsilon=(v,\text{Cyl}_e)$ be the corresponding edge of $S^c$. The above claim shows that any element $g\in G$ that fixes $\epsilon$ must fix $e$, so $g=1$. This proves that $\epsilon$ has trivial stabilizer for the $G$-action on $U$, which concludes the proof of the proposition. 
\end{proof}

\paragraph*{End of the proof of Theorem~\ref{at-amen}.}

From now on, we assume that $(G,\calf)$ satisfies the assumptions from Theorem~\ref{at-amen}.

Recall that a \emph{$\calz$-splitting} of $(G,\calf)$ is a minimal, simplicial $(G,\calf)$-tree whose edge stabilizers are either trivial, or cyclic and nonperipheral. The \emph{$\calz$-splitting graph} 
$\ZS$
is the simplicial graph whose vertices are the $\calz$-splittings of $(G,\calf)$, where two splittings are joined by an edge if they are compatible. Its hyperbolicity was proved by Mann \cite{Man} for free groups, and extended to the context of free products in \cite{Hor3}. Since every optimal folding path projects to an unparameterized quasigeodesic in $\ZS$, there exists $R>0$ such that if $S_t$ is an optimal folding path and $t_1<t_2<t_3<t_4$ and $d_{\ZS}(S_{t_2},S_{t_3})>R$, then $S_{t_1}$ and $S_{t_4}$ are not compatible with any common free splitting.

Given an arational tree $T$, an optimal morphism $f:S\to T$ and positive real numbers $t<t'$, we let $[t,t']_f$ be the collection of all simplices $\Delta\in\Simp$ such that there exists a tree $S'\in\tilde \Delta$ through which $f$ factors, with $e^{-t'}\leq \vol(S'/G)\leq e^{-t}$. 
By Proposition~\ref{sphere-finite} and Remark~\ref{rk-covol}, the set $[t,t']_f$ is finite. Given $n\in\mathbb{N}$, we let $m_n(f)$ be the smallest integer  
such that the projection to $\ZS$ of the set $[n,n+m_n(f) ]_f$ has diameter at least $R$. Existence of $m_n(f)$ is justified by the following lemma.

\begin{lemma}
For all $n\in\mathbb{N}$, we have $m_n(f)<+\infty$.
\end{lemma} 

\begin{proof} 
Since the range of $f$ is arational, it follows from the description of the Gromov boundary of $\ZS$ in terms of $\calz$-averse trees given in \cite{Hor3}, together with the fact that arational trees are $\calz$-averse \cite[Proposition~4.7]{Hor2}, that every folding path guided by $f$ projects to an infinite unparameterized quasigeodesic ray in $\ZS$. Therefore, by choosing $m$ sufficiently large, we can ensure that $[k,k+m]_f$ contains two simplices whose projections to $\ZS$ are at distance larger than $R$ from one another, as required.
\end{proof}

Let $M_n(f):=\max_{n\leq k\leq 2n} m_k(f)$. 
By definition of $R$ and Proposition~\ref{stab-pair}, this implies that for all $t\in [n,2n]$, the set $[t,t+M_n(f)]_f$ contains two simplices 
whose common stabilizer is amenable.

\begin{proof}[Proof of Theorem~\ref{at-amen}]
The same argument as in the proof of Lemma~\ref{sphere-borel} shows that for all $n,m$, the set $[n,n+m]_f$ depends measurably on $f\in\Opt$. Therefore, for all $n\in\mathbb{N}$, the integers $m_n(f)$ and $M_n(f)$ depend measurably on $f$. Let $\calf^{\amen}(\Simp)$ be the countable collection of all finite sets in $\Simp$ whose stabilizer in $\Out(G,\calf^{(\mathrm{t})})$ is amenable. For all $n\in\mathbb{N}$ and all $f\in\Opt$, we then let $\mu_n(f)$ be the probability measure on $\calf^{\amen}(\Simp)$ defined as $$\mu_n(f):=\frac{1}{n}\int_n^{2n}\delta_{[t,t+M_n(f)]_f} d\text{Leb}(t),$$ where $\delta_{[t,t+M_n(f)]_f}$ is the Dirac measure on the set $[t,t+M_n(f)]_f$. Then the map $f \mapsto \mu_n(f)$ is Borel. Arguing as in the proof of Proposition~\ref{name-enough}, we can then associate to every tree $T\in\AT$ a sequence of probability measures $\mu_n(T)$ on $\calf^{\amen}(\Simp)$, depending measurably on $T$, so that for all $n\in\mathbb{N}$ and all $\Phi\in\Out(G,\calf^{(\mathrm{t})})$, we have $$||\mu_n(\Phi.T)-\Phi.\mu_n(T)||_1\to 0$$ as $n$ goes to $+\infty$. Since the stabilizer of every point in $\calf^{\amen}(\Simp)$ is amenable, one can apply Proposition~\ref{amen} with $K:=\calf^{\amen}(\Simp)$, and deduce that the $\Out(G,\calf^{(\mathrm{t})})$-action on $\AT$ is Borel amenable.
\end{proof}

\bibliographystyle{amsplain}
\bibliography{amenability-bib}

\begin{flushleft}
Mladen Bestvina\\
Department of Mathematics\\ 
University of Utah\\ 
155 South 1400 East, JWB 233\\ 
Salt Lake City, Utah 84112-0090, United States\\
\emph{e-mail:}\texttt{bestvina@math.utah.edu}
\\[8mm]

Vincent Guirardel\\
Institut de Recherche Math\'ematique de Rennes\\
Universit\'e de Rennes 1 et CNRS (UMR 6625)\\
263 avenue du G\'en\'eral Leclerc, CS 74205\\
F-35042  RENNES C\'edex\\
\emph{e-mail:}\texttt{vincent.guirardel@univ-rennes1.fr}\\[8mm]

Camille Horbez\\
CNRS\\ 
Laboratoire de Math\'ematique d'Orsay\\
Universit\'e Paris-Sud, Universit\'e Paris-Saclay\\ 
F-91405 ORSAY\\
\emph{e-mail:}\texttt{camille.horbez@math.u-psud.fr}
\\[8mm]

\end{flushleft}

\end{document}